\documentclass[reqno,12pt,final]{amsart}
\usepackage{amsmath,amsthm,amsfonts,amssymb}
\usepackage[english]{babel}
\usepackage{graphicx}

\newcommand{\ee}{\varepsilon}

\newtheorem{theorem}{Theorem}
\newtheorem{lemma}{Lemma}
\newtheorem{remark}{Remark}

\allowdisplaybreaks
\input{tcilatex}
\DeclareMathOperator* {\essmin}{ess\min}
\DeclareMathOperator* {\essmax}{ess\max}

\begin{document}
\date{}

\begin{center}
\textbf{{Homogenization of Dynamic Signorini-Type Problems on Critically
Oscillating Boundaries under Time-Periodic Forcing} }
\end{center}

\vspace{0.5cm}

\begin{center}
\textbf{J.I. D\'{\i}az, A.V. Podolskyi, T.A. Shaposhnikova }
\end{center}

\vspace{1cm}

We study the homogenization, in the critical scaling regime, of a boundary
value problem with a nonlinear dynamic Signorini-type condition posed on a
rapidly oscillating portion of the boundary. The source term is
time-periodic and we look for time-periodic solutions. Using the method of
oscillating test functions (Tartar), compactness, and monotonicity
arguments, we identify the homogenized problem and the effective nonlinear
boundary operator. In contrast with the evolutionary (initial value)
setting, the periodic framework eliminates memory effects and yields an
instantaneous time-periodic operator defined through a periodic-in-time cell
problem. 


\textit{\textrm{\vspace{0.5cm} }}

\section{Introduction}

The main goal of this paper is to study a dynamic Signorini-type problem
with unilateral boundary conditions posed on a part of the boundary which is
rapidly oscillating. The source term is time-periodic and we look for
time-periodic solutions. There are multiple physical problems leading to
these type of formulations. For instance, the model under consideration in
this paper arises in some simplifications related to nano-composite
membranes and nano-osmosis (in the spirit of the paper \cite{DGCPSh-5-Osmo}%
). In this framework, $u$ represents pressure, concentration, or chemical
potential, the oscillating boundary models nanopores, Signorini-type
conditions model selective permeability.

The problem consists of the homogenization of a time-periodic Poisson
equation with a time-periodic source term, where the boundary condition
rapidly alternates between a dynamic Signori-type boundary condition and a
pure Neumann boundary condition. This type of equation models systems where
macroscopic, stable phenomena interact with microstructures that oscillate
rapidly in space (on the boundary of the geometrical domain). The systems
are subjected to coupled dynamics and sources that are periodic in time.
Some other prominent physical and engineering contexts for this problem
include: a) \textit{Fuel Cells and Lithium-Ion Batteries.} Porous electrode
frontiers usually present a rough or oscillating geometry at the microscopic
scale. Chemical reactions at the interface are neither instantaneous nor
static, they are modeled via dynamic Signorini-type boundary conditions
(which involve time derivatives at the boundary due to charge or mass
accumulation). If the device operates under alternating current or
charge/discharge cycles, both the source and the boundary conditions acquire
a time-periodic nature; b) \textit{Heat Transfer with Boundary Thermal
Capacitance.} Consider a solid with a rough surface exposed to a cycling
thermal flux (e.g., day/night cycles or pistons in an engine). If the thin
surface layer significantly absorbs/stores heat, the traditional Neumann
boundary condition transforms into a dynamic condition (Robin-dynamic type);
c) \textit{Microfluidic Devices and Chemical Sensors}\textbf{.} Channels
where the walls have active patches (catalyzed periodic chemical reactions)
interleaved with insulating patches (Neumann conditions).

Let us mention now some reasons to consider the time-periodic formulation in
some of those applications. \textit{Electrochemical Impedance Spectroscopy
(EIS) and Batteries: }This is currently the most direct field of
application, where the homogenization techniques provide significant
industrial value. The phenomenon is to determine the state of health of a
lithium-ion battery or a fuel cell, a low-amplitude alternating current with
a fixed period $T=2\pi /\omega $ is injected into the system. As a
simplified model, we can consider the Poisson equation governing the
electric potential or ion concentration. The source (the injected current)
is $T$-periodic. The reason to consider some dynamic oscillating boundary is
that the electrodes are porous and rough at the microscopic scale $%
\varepsilon $. At the electrode surface, the double capacitive layer occurs
(which stores charge, generating the dynamic condition $\partial _{t}u$)
alongside the charge transfer reaction in active patches, while the
insulating zones act as a Neumann condition. Finally, the relevance of the
homogenization limit is that the appearance of some \textquotedblleft
strange term\textquotedblright \ non-local in time predicted by mathematical
theory translates physically into the effective impedance of the electrode.
This explicitly explains how microscopic roughness alters the global
resistance and capacitance of the battery (see, e.g. \cite{Schumuck}, and 
\cite{DGRLithium}).

In the case of \  \textit{Heat Conduction in Engines and Bioclimatic
Architecture, }some\textit{\ }natural or mechanical thermal cycles impose a
strict temporal periodicity. In \textit{daily cycles} (Climatization and
Geothermal Energy), the outdoor temperature or solar radiation on the rough
facade of a building (or the ground) varies with a fixed period of $T=24%
\text{ hours and t}$he heat source is essentially periodic. In the case of
Internal Combustion Engine Cycles, the walls of the cylinders in an engine
receive intermittent heat pulses due to injection and combustion strokes ($T$
depends on the engine's RPM). The dynamic oscillating boundary is justified
because the outer surface has microscopic cooling fins or rough textures to
dissipate heat, and these textures have their own thermal capacity (storing
heat before transmitting it), the boundary condition is dynamic. Upon
applying spatial homogenization, the 24-hour period (or the RPM cycle)
remains fixed, but the strange term reveals the effective thermal inertia
that the rough material will exhibit macroscopically (see \cite{Fudym}).

Finally, the case of \textit{Microfluidic Devices and Organs-on-a-Chip
(Lab-on-a-Chip) }is very relevant in modern bioengineering where microscopic
channels are used to transport biological fluids or chemical reagents via
piezoelectric or peristaltic pumping. The interesting phenomenon is that
microfluidic pumps operate in a pulsatile manner, generating pressure fields
or solute concentrations with a fixed temporal periodicity $T$ (the pump
cycle). The dynamic oscillating boundary is justified since the walls of
these channels are often decorated with patches of chemical catalysts or
biological receptors (active patches with dynamic adsorption/reaction
conditions) interleaved with inert plastic (Neumann condition). The
relevance of the limit (the spatial homogenization) allows chip designers to
determine the total percentage of reagent that will be absorbed per pumping
cycle $T$, eliminating the need to computationally simulate millions of
microscopic patches inside the channel (see, e.g., \cite{Auton}).

Thus, time periodicity corresponds to cyclic operational regimes:
oscillating pressure, alternating electric fields, periodic concentration
gradients. Although, we will present the technical details on the domain in
the next section, we outline now that the problem under consideration can be
simply formulated in the following terms. Given a time period $T>0,$ in the
cylinder $Q^{\infty }=\Omega \times \mathbb{R}$, and assuming a source term $%
f\in H^{1}(\mathbb{R},L^{2}(\Omega ))$ is $T$-time-periodic, i.e. 
\begin{equation*}
f(x,t) = f(x, t+T)\text{ for any }t\in \mathbb{R}\text{, a.e. }x\in \Omega .
\end{equation*}%
Notice that we can also use a different natural notation within the time
torus $\mathbb{T}=\mathbb{R}/T\mathbb{Z}$, where $T$ is the fixed period,
and then, replacing the above condition, $f\in H^{1}(\mathbb{R},L^{2}(\Omega
))$ $T$-time-periodic, by using the Bochner spaces integrated over time,
such as $L^{2}(\mathbb{T};H^{1}(\Omega ))$ (and as it will appear later $%
H^{1}(\mathbb{T};L^{2}(\partial \Omega ))$). We look for time-periodic
solutions $u_{\varepsilon }(x,t)$ of the Poisson problem with a dynamic
Signorini-type boundary condition on the rapidly oscillating portion of the
boundary $l_{\varepsilon }^{\infty }=l_{\varepsilon }\times \mathbb{R}$ 
\begin{equation}
(PP)\left \{ 
\begin{array}{lr}
-\Delta _{x}u_{\varepsilon }(x,t) = f(x,t), & (x,t)\in Q^{\infty }, \\ 
u_{\varepsilon }\geq 0,\beta (\varepsilon )\partial _{t}u_{\varepsilon
}+\partial _{\nu }u_{\varepsilon }+\beta (\varepsilon )\sigma
(u_{\varepsilon })\geq 0, &  \\ 
u_{\varepsilon }(\beta (\varepsilon )\partial _{t}u_{\varepsilon }+\partial
_{\nu }u_{\varepsilon }+\beta (\varepsilon )\sigma (u_{\varepsilon }))=0, & 
(x,t)\in l_{\varepsilon }^{\infty }=l_{\varepsilon }\times \mathbb{R}, \\ 
\partial _{\nu }u_{\varepsilon }=0, & (x,t)\in \gamma _{\varepsilon
}^{\infty }=\gamma _{\varepsilon }\times \mathbb{R}, \\ 
u_{\varepsilon }(x,t)=u_{\varepsilon }(x,t+T)\text{ } & \text{for any }t\in 
\mathbb{R}\text{, }x\in l_{\varepsilon }, \\ 
u_{\varepsilon }(x,t)=0, & (x,t)\in \Gamma _{1}^{\infty }=\Gamma _{1}\times 
\mathbb{R},%
\end{array}%
\right.
\end{equation}%
where $l_{\varepsilon }\cup \gamma _{\varepsilon }=\Gamma _{2}$, $\partial
\Omega =\Gamma _{1}\cup \Gamma _{2}$, $\beta (\varepsilon )=\exp (\alpha
^{2}/\varepsilon )$, $\alpha >0$, $\nu $ is the unit outward normal vector
to $l_{\varepsilon }$. The function $\sigma :\mathbb{R}\rightarrow \mathbb{R}
$ is a continuously differentiable function, such that $\sigma (0)=0$ and $%
0<k_{1}\leq \sigma ^{^{\prime }}(u)\leq k_{2}$ for an arbitrary $u\in 
\mathbb{R},$ and some constants $k_{1}, k_{2}>0$. We remark that on the 
\textit{microscopic boundary }$\Gamma _{2}=\partial \Omega \setminus \Gamma_1$%
, we are assuming a structure starting with a periodic reference cell $%
Y=(0,1)^{d-1}$ with a subset $S\subset Y$ representing the solid part. The
oscillatory boundary $\Gamma _{2}$ is generated by scaling and repetition of 
$S$ with characteristic size $\varepsilon $ and critical thickness so that
surface effects persist in the limit. We recall that, as explained in the
monograph \cite{DGSH2Book}, this choice of the parameter $\beta (\varepsilon
)$ corresponds to the more interesting case: the so called \textit{Critical
Case}\textbf{. }It can be shown\textbf{\ }that there are two other different
cases. The \textit{Subcritical Case}\textbf{\ }corresponds to when the
dynamic zone is small in measure, and in the limit the entire boundary
behaves as a pure Neumann boundary. The \textit{Supercritical Case}\textbf{\ 
}corresponds to when the dynamic zone dominates, the homogenized boundary
fully absorbs the dynamic condition. Finally, the \textit{Critical Case }%
arises when both zones are perfectly balanced, and passing to the limit $%
\varepsilon \rightarrow 0$ in space generates a \textit{modified effective
boundary condition}. \ It implies the emergence of \textit{\textquotedblleft
strange terms\textquotedblright \ (}in our case \textit{strange operators)}
arising from dynamic and non-linear boundary conditions on critical
microstructures. Due to the time derivative at the original boundary, $%
\partial _{t}u$, the final homogenized boundary acquires a \textquotedblleft
reactive capacitance\textquotedblright \ or a damped memory term. Since the
source is time-periodic and the operator is linear, the limit solution $%
u_{0}(x,t)$ preserves exactly the same period $T$, but its spatial profile
is dictated by the homogenized coefficients that capture the microstructural
geometry of the boundary.

In order to study the periodical problem $(PP)$ we will start by considering
the associate initial value problem in the cylinder $Q^{T}=\Omega \times
(0,T)$, 
\begin{equation}
(IVP)\left \{ 
\begin{array}{lr}
-\Delta _{x}u_{\varepsilon }(x,t)=f(x,t), & (x,t)\in Q^{T}, \\ 
u_{\varepsilon }\geq 0,\beta (\varepsilon )\partial _{t}u_{\varepsilon
}+\partial _{\nu }u_{\varepsilon }+\beta (\varepsilon )\sigma
(u_{\varepsilon })\geq 0, &  \\ 
u_{\varepsilon }(\beta (\varepsilon )\partial _{t}u_{\varepsilon }+\partial
_{\nu }u_{\varepsilon }+\beta (\varepsilon )\sigma (u_{\varepsilon }))=0, & 
(x,t)\in l_{\varepsilon }^{T}=l_{\varepsilon }\times (0,T), \\ 
\partial _{\nu }u_{\varepsilon }=0, & (x,t)\in \gamma _{\varepsilon
}^{T}=\gamma _{\varepsilon }\times (0,T), \\ 
u_{\varepsilon }(x,0)=u^{0}(x), & x\in l_{\varepsilon }, \\ 
u_{\varepsilon }(x,t)=0, & (x,t)\in \Gamma _{1}^{T}=\Gamma _{1}\times (0,T),%
\end{array}%
\right.
\end{equation}%
where $u^{0}\in H_{0}^{1}(-l,l)$, $u^{0}(x)\geq 0$. We will show the
existence of solutions to the periodic problem $(PP)$ by getting a fixed
point function $u^{0}$ of the Poincar\'{e} map%
\begin{equation*}
F: u^{0}(\cdot)\rightarrow u_{\varepsilon }(\cdot, T)\text{ from }%
L^{2}(l_{\varepsilon })\text{ into }L^{2}(l_{\varepsilon }).
\end{equation*}

We will show the existence and uniqueness of $u_{\varepsilon }^{T}(x,t),$ $T$%
-time-periodic solution of $(PP)$ by two different methods. First, we check
the assumptions of some abstract results (see \cite{Kenmochi1981}) to the
case of our formulation. Moreover, if we assume 
\begin{equation*}
f\in H^{1}(\mathbb{R},L^{2}(\Omega ))\text{, }T-\text{time-periodic},
\end{equation*}%
we can get some useful estimates on the $T$-time-periodic solution $%
u_{\varepsilon }^{T}(x,t)$ of $(PP).$ Indeed, we can define 
\begin{equation*}
f_{1}(x):=\underset{t\in \lbrack 0,T]}{\essmin}f(x,t)\leq f_{2}(x):=\underset%
{t\in \lbrack 0,T]}{\essmax}f(x,t),
\end{equation*}%
(see Lemma 5) and since $f_{i}\in L^{2}(\Omega )$, for $i=1,2$, we can
define the (unique) solutions of the stationary problems 
\begin{equation}
(SP)\left \{ 
\begin{array}{lr}
-\Delta _{x}u_{\varepsilon }^{i}(x)=f_{i}(x), & x,\in \Omega , \\ 
u_{\varepsilon }^{i}\geq 0,\partial _{\nu }u_{\varepsilon }^{i}+\beta
(\varepsilon )\sigma (u_{\varepsilon }^{i})\geq 0, &  \\ 
u_{\varepsilon }^{i}(\partial _{\nu }u_{\varepsilon }^{i}+\beta (\varepsilon
)\sigma (u_{\varepsilon }^{i}))=0, & x\in l_{\varepsilon }, \\ 
\partial _{\nu }u_{\varepsilon }^{i}=0, & x\in \gamma _{\varepsilon }, \\ 
u_{\varepsilon }^{i}(x)=0, & x\in \Gamma _{1}.%
\end{array}%
\right. 
\end{equation}%
Then, we will prove that%
\begin{equation*}
u_{\varepsilon }^{1}(x)\leq u_{\varepsilon }^{T}(x,t)\leq u_{\varepsilon
}^{2}(x)\text{ for any }t\in \mathbb{R},\text{ on }l_{\varepsilon }.
\end{equation*}

The main result of this paper is to prove that $u_{\varepsilon
}^{T}(x,t)\rightharpoonup {u}_{0}^{T}(x,t)$ as $\varepsilon \rightarrow 0$
with ${u}_{0}^{T}(x,t)$ a $T-$periodic solution of the homogenized problem
of $(PP),$ given by the stationary problem (depending $T-$periodically on $t,
$ as a parameter). When passing to the limit, as $\varepsilon \rightarrow 0,$
only in the spatial heterogeneity and not in the temporal periodicity $T$,
the temporal frequency remains fixed at a macroscopic scale. In this
scenario, time essentially acts as a continuous parameter during the spatial
limit process. We will prove that $u_{\varepsilon }^{T}(x,t)\rightharpoonup {%
u}_{0}^{T}(x,t)$ with ${u}_{0}^{T}(x,t)$ $T-$periodic solution of the
homogenized problem of $(PP),$ given by the stationary problem (dependinhg
on $t$ as a parameter)%
\begin{equation}
(PP)_{Hom}\left \{ 
\begin{array}{lr}
-\Delta _{x}{u}_{0}(x,t)=f(x,t), & (x,t)\in Q^{\infty }, \\ 
-\partial _{x_{2}}u_{0}+\mathcal{M}u_{0}=\mathcal{M}H_{u_{0}}, & (x,t)\in
\Gamma _{2}^{\infty }, \\ 
H_{u_{0}}\geq 0,\, \partial _{t}H_{u_{0}}+\mathcal{L}H_{u_{0}}+\sigma
(H_{u_{0}})\geq \mathcal{L}u_{0}, & (x,t)\in \Gamma _{2}^{\infty }, \\ 
H_{u_{0}}(\partial _{t}H_{u_{0}}+\mathcal{L}H_{u_{0}}+\sigma (H_{u_{0}})-%
\mathcal{L}u_{0})=0, & (x,t)\in \Gamma _{2}^{\infty }, \\ 
H_{u_{0}}(x,t)=H_{u_{0}}(x,t+T), & \text{for any }t\in \mathbb{R}\text{, }%
x\in \Gamma _{2}, \\ 
u_{0}(x,t)=0, & (x,t)\in \Gamma _{1}^{\infty },%
\end{array}%
\right. 
\end{equation}%
where \textit{\textrm{$\mathcal{M}$}} $=\frac{\pi }{\alpha ^{2}}$, \textit{%
\textrm{$\mathcal{L}$}}$=\frac{\pi }{2C_{0}l_{0}\alpha ^{2}}$. While keeping
time as a continuous parameter, the time derivative at the boundary, i.e. $%
\partial _{t}u$, propagates into the homogenized domain or boundary.
Crucially, it does not appear as a simple derivative, but rather as a 
\textit{memory operator or a non-local term in time }$H_{u_{0}}$. This
memory operator simplifies into a \textit{compact periodic operator} (or an
impedance matrix) that couples the harmonics of the source with the
geometric response of the boundary.

To carry out such a program, previously, we will prove that $u_{\varepsilon
}(x,t)\rightharpoonup {u}_{0}(x,t)$ as $\varepsilon \rightarrow 0$ with ${u}%
_{0}(x,t)$ a weak solution of the homogenized problem of $\ $the $(IVP),$
given by the stationary problem (dependinhg on $t$ as a parameter) 
\begin{equation}
(IVP)_{Hom}\left \{ 
\begin{array}{lr}
-\Delta _{x}{u}_{0}(x,t)=f(x,t), & (x,t)\in Q^{T}, \\ 
-\partial _{x_{2}}u_{0}+\mathcal{M}u_{0}=\mathcal{M}H_{u_{0}}, & (x,t)\in
\Gamma _{2}^{T}, \\ 
H_{u_{0}}\geq 0,\, \partial _{t}H_{u_{0}}+\mathcal{L}H_{u_{0}}+\sigma
(H_{u_{0}})\geq \mathcal{L}u_{0}, & (x,t)\in \Gamma _{2}^{T}, \\ 
H_{u_{0}}(\partial _{t}H_{u_{0}}+\mathcal{L}H_{u_{0}}+\sigma (H_{u_{0}})-%
\mathcal{L}u_{0})=0, & (x,t)\in \Gamma _{2}^{T}, \\ 
H_{u_{0}}(x,0)=u^{0}(x), & x\in \Gamma _{2}, \\ 
u_{0}(x,t)=0, & (x,t)\in \Gamma _{1}^{T},%
\end{array}%
\right. 
\end{equation}%
with\textit{\textrm{\ $\mathcal{M}$}} and \textit{\textrm{$\mathcal{L}$ }}as
before. Thus the homogenized problem $(PP)_{Hom}$ reaches a stabilized
cyclic regime. The periodic-in-time homogenization of dynamic Signorini-type
conditions $(PP)_{Hom}$ produces an effective nonlinear boundary operator
encoding microscopic \textit{cyclic dynamics,} in contrast with the
homogenized initial value problem $(IVP)_{Hom}$, initial memory effects
disappear and are replaced by a stabilized periodic response.

The organization of the paper is as follows: Section 2 deals with a detailed
study of the initial problem which is divided into subsections: the
statement of the problem, the main theorem about the homogenization for the
initial problem with its detailed proof. Section 3 deals with the T-periodic
formulation and its also divided in subsections: the existence of time
periodic solutions of the microscopic problem and, finally, the
homogenization result.


\section{On the initial problem}

\subsection{Statement of the problem.}

Let $\Omega $ be a bounded domain in $\mathbb{R}^{2}\cap \{x_{2}>0\}$ with a
smooth boundary, consisting of two parts $\partial \Omega =\Gamma _{1}\cup
\Gamma _{2}$, where $\Gamma _{1}=\partial \Omega \cap \{x_{2}>0\}$ and $%
\Gamma _{2}=\partial \Omega \cap \{x_{2}=0\}=[-l,l]$ for some $l>0$. We
define 
\begin{gather*}
Y_{1}=\{(y_{1},0):-1/2<y_{1}<1/2\}, \\
\hat{l}_{0}=\{(y_{1},0):-l_{0}<y_{1}<l_{0}\} \subset Y_{1},\, \,l_{0}\in
(0,1/2).
\end{gather*}

For a small parameter $\varepsilon >0$ and a parameter $0<a_{\varepsilon
}\ll \varepsilon $ whose value is \textquotedblleft
critical\textquotedblright , that is $a_{\varepsilon }=C_{0}\varepsilon \exp
(-\alpha ^{2}/\varepsilon )$, $C_{0}$ and $\alpha$ are positive constants,
we introduce the sets\textit{\textrm{\ 
\begin{equation*}
\widetilde{G_{\varepsilon }}=\bigcup \limits_{j\in \mathbb{Z}^{^{\prime
}}}(a_{\varepsilon }\hat{l_{0}}+{\varepsilon }j)=\bigcup \limits_{j\in 
\mathbb{Z}^{^{\prime }}}l_{\varepsilon }^{j},
\end{equation*}%
}}where\textit{\textrm{\ $\mathbb{Z}^{^{\prime }}=\mathbb{Z}\times \{0\}$ }}%
is the set of vectors of the form \textit{\textrm{$i=(j_{1},0),\,j_{1}\in 
\mathbb{Z}$. }}We set\textit{\textrm{\ 
\begin{equation*}
\Upsilon _{\varepsilon }=\{j\in \mathbb{Z}^{^{\prime }}:\overline{%
l_{\varepsilon }^{j}}\subset \lbrack -l+2\varepsilon ,l-2\varepsilon ]\times
\{0\} \}.
\end{equation*}%
}}Next, we define the sets\textit{\textrm{\ 
\begin{equation*}
Y_{\varepsilon }^{j}=\varepsilon {Y}_{1}+\varepsilon {j},\,j\in \mathbb{Z}%
^{^{\prime }},\, \,l_{\varepsilon }=\bigcup \limits_{j\in \Upsilon
_{\varepsilon }}l_{\varepsilon }^{j}.
\end{equation*}%
}}It is easy to see that\textit{\textrm{\ $\overline{l_{\varepsilon }^{j}}%
\subset Y_{\varepsilon }^{j}$. }}

We introduce the set \textit{\textrm{$\gamma _{\varepsilon }=\Gamma
_{2}\setminus \overline{l_{\varepsilon }}$. }}Note that\textit{\textrm{\ $%
|l_{\varepsilon }^{j}|=2a_{\varepsilon }l_{0}$ }}for an arbitrary\textit{%
\textrm{\ $j\in \mathbb{Z}^{^{\prime }}$ }}and\textit{\textrm{\ $%
|l_{\varepsilon }|\cong d{a}_{\varepsilon }{\varepsilon }^{-1}$, $d=$}}$%
const>0$.\textit{\textrm{\ }}In the cylinder\textit{\textrm{\ $Q^{T}=\Omega
\times (0,T)$, }}we consider the problem 
\begin{equation}
\left \{ 
\begin{array}{lr}
-\Delta _{x}u_{\varepsilon }(x,t)=f(x,t), & (x,t)\in Q^{T}, \\ 
u_{\varepsilon }\geq 0,\beta (\varepsilon )\partial _{t}u_{\varepsilon
}+\partial _{\nu }u_{\varepsilon }+\beta (\varepsilon )\sigma
(u_{\varepsilon })\geq 0, &  \\ 
u_{\varepsilon }(\beta (\varepsilon )\partial _{t}u_{\varepsilon }+\partial
_{\nu }u_{\varepsilon }+\beta (\varepsilon )\sigma (u_{\varepsilon }))=0, & 
(x,t)\in l_{\varepsilon }^{T}=l_{\varepsilon }\times (0,T), \\ 
\partial _{\nu }u_{\varepsilon }=0, & (x,t)\in \gamma _{\varepsilon
}^{T}=\gamma _{\varepsilon }\times (0,T), \\ 
u_{\varepsilon }(x, 0) = u^{0}(x), & x\in l_{\varepsilon }, \\ 
u_{\varepsilon }(x,t) = 0, & (x,t)\in \Gamma _{1}^{T}=\Gamma _{1}\times
(0,T),%
\end{array}%
\right.  \label{1}
\end{equation}%
where $\beta (\varepsilon )=\exp (\alpha ^{2}/\varepsilon )$, $\alpha >0$%
\textit{\textrm{, $f\in H^{1}(0,T; L^{2}(\Omega ))$, $u^{0}\in
H_{0}^{1}(-l,l) $, $u^{0}(x)\geq 0$, $\nu $ }}is the unit outward normal
vector to \textit{\textrm{$l_{\varepsilon }^{T}$. }}The function \textit{%
\textrm{$\sigma :\mathbb{R}^{1}\rightarrow \mathbb{R}^{1}$ is a
differentiable }}continuous function, $\sigma (0)=0$, and there exist
constants\textit{\textrm{\ $k_{1},k_{2}>0$ }}such that\textit{\textrm{\ $%
0<k_{1}\leq \sigma ^{^{\prime }}(u)\leq k_{2}$ }}for an arbitrary \textit{%
\textrm{$u\in \mathbb{R}^{1}$. }}

We define the convex closure sets \textit{\textrm{%
\begin{gather*}
\mathcal{K}_{\varepsilon }=\{ \phi \in H^{1}(\Omega ,\Gamma _{1})\,|\, \phi
\geq 0\, \mbox{on}\, \,l_{\varepsilon }\}, \\
\mathbb{K}_{\varepsilon }=\{v\in L^{2}(0,T;H^{1}(\Omega ,\Gamma
_{1})):v(\cdot ,t)\in \mathcal{K}_{\varepsilon },\, \mbox{a.e.}\, \,t\in
(0,T)\}.
\end{gather*}%
}}Here,\textit{\textrm{\ $H^{1}(\Omega ,\Gamma _{1})$ is the closure in $%
H^{1}$-norm of the set of infinitely differentiable functions vanishing near
the boundary $\Gamma _{1}$. }}

\textit{\textrm{Given $f\in H^{1}(0,T;L^{2}(\Omega))$, we say that $%
u_{\varepsilon}\in \mathbb{K}_{\varepsilon}$ is a strong solution to %
\eqref{1}, if $\partial_{t}u_{\varepsilon}\in
L^{2}(0,T;L^{2}(l_{\varepsilon}))$, $u_{\varepsilon}(x,0)=u^{0}(x)$ for $%
x\in l_{\varepsilon}$, and it satisfies the variational inequality 
\begin{equation}  \label{2}
\begin{gathered} \beta(\varepsilon)\int
\limits_{l^{T}_{\ee}}\partial_{t}u_{\ee}(\phi-u_{\ee})dx_{1}dt+ \int
\limits_{Q^{T}}\nabla{u_{\ee}}\nabla(\phi-u_{\ee})dx{dt}+\beta(\varepsilon)%
\int \limits_{l^{T}_{\ee}}\sigma(u_{\ee})(\phi-u_{\ee})dx_{1}dt\ge \\ \ge
\int \limits_{Q^{T}}f(\phi-u_{\ee})dx{dt}, \end{gathered}
\end{equation}
for an arbitrary function $\phi \in \mathbb{K}_{\varepsilon}$. }}

\begin{theorem}
\textit{\textrm{For any $\varepsilon>0$ the problem \eqref{1} has a unique
strong solution $u_{\varepsilon}$. Moreover, $u_{\varepsilon}$ satisfies the
following estimates 
\begin{equation}  \label{3}
\begin{gathered} \beta(\varepsilon)\max
\limits_{[0,T]}\Vert{u}_{\ee}\Vert^{2}_{L^{2}(l_{\ee})}+\Vert{u_{\ee}}%
\Vert^{2}_{L^{2}(0,T;H^{1}(\Omega,\Gamma_{1}))}\le
K(\Vert{f}\Vert^{2}_{L^{2}(Q^{T})}+\Vert{u}^{0}\Vert^{2}_{L^{2}(-l,l)}),\\
\beta(\varepsilon)\Vert \partial_{t}u_{\ee}\Vert^{2}_{L^{2}(l^{T}_{\ee})} +
\max \limits_{[0,T]}\Vert \nabla{u}_{\ee}\Vert^{2}_{L^{2}(\Omega)}\\ \le
K\Bigl(\max \limits_{[0,T]}\Vert{f}(\cdot,t)\Vert^{2}_{L^{2}(\Omega)} +
\Vert{f}\Vert^{2}_{H^{1}(0,T;L^{2}(\Omega))}+\Vert{u}^{0}%
\Vert^{2}_{H^{1}_{0}(-l,l)}\Bigr). \end{gathered}
\end{equation}
}}
\end{theorem}

\begin{proof}
\textit{\textrm{For $\delta>0$, we consider the problem 
\begin{equation}  \label{4}
\begin{gathered} \left \{ \begin{array}{lr}
-\Delta{u^{\delta}_{\ee}}=f(x,t), & (x,t)\in Q^{T},\\
\partial_{\nu}u^{\delta}_{\ee}+\beta(\varepsilon)\partial_{t}u^{\delta}_{%
\ee}+\beta(\varepsilon)\sigma(u^{\delta}_{\ee})=-\beta(\varepsilon)%
\delta^{-1}(u^{\delta}_{\ee})^{-}, & (x,t)\in l^{T}_{\ee},\\
u^{\delta}_{\ee}(x,0)=u^{0}(x), & x\in l_{\ee},\\ u^{\delta}_{\ee}(x,t)=0, &
(x,t)\in \Gamma^{T}_{1}, \end{array}\right. \end{gathered}
\end{equation}
where $u^{-}$ is a negative part of the function $u$, i.e. $u^{-}=\inf(0, u)$%
. }}

\textit{\textrm{By the strong solution of the problem \eqref{4}, we call a
function $u^{\delta}_{\varepsilon}\in L^{2}(0,T; H^{1}(\Omega,\Gamma_{1}))$
such that $\partial_{t}u^{\delta}_{\varepsilon}\in
L^{2}(0,T;L^{2}(l_{\varepsilon}))$, $u^{\delta}_{\varepsilon}(x,0)=u^{0}(x)$
for $x\in l_{\varepsilon}$, and satisfying the integral identity 
\begin{gather}
\beta(\varepsilon)\int
\limits_{l^{T}_{\varepsilon}}\partial_{t}u^{\delta}_{\varepsilon}\phi{dx_{1}}%
dt+\int \limits_{Q^{T}}\nabla{u}^{\delta}_{\varepsilon}\nabla \phi{dx}dt+
\beta(\varepsilon)\int
\limits_{l^{T}_{\varepsilon}}\sigma(u^{\delta}_{\varepsilon})\phi{d}x_{1}dt+
\notag \\
+\beta(\varepsilon)\delta^{-1}\int
\limits_{l^{T}_{\varepsilon}}(u^{\delta}_{\varepsilon})^{-}\phi{dx_{1}}%
dt=\int \limits_{Q^{T}}f\phi{dx}dt,  \label{5}
\end{gather}
where $\phi$ is an arbitrary function $\phi \in
L^{2}(0,T;H^{1}(\Omega,\Gamma_{1}))$. }}

\textit{\textrm{Taking into account that $\sigma(w) + \delta^{-1}w^{-}$ is a
monotone function of $w\in \mathbb{R}^{1}$, we conclude that problem %
\eqref{4} has a unique strong solution. In addition, we have the estimate 
\begin{gather}
\beta(\varepsilon)\max \limits_{[0,T]}\Vert
u^{\delta}_{\varepsilon}\Vert^{2}_{L^{2}(l_{\varepsilon})}+ \Vert
\nabla_{x}u^{\delta}_{\varepsilon}\Vert^{2}_{L^{2}(Q^T)}+
\beta(\varepsilon)\delta^{-1}\Vert(u^{\delta}_{\varepsilon})^{-}%
\Vert^{2}_{L^{2}(l^T_{\varepsilon})}\le  \notag \\
\le K(\Vert{f}\Vert^{2}_{L^{2}(Q^{T})}+\Vert{u}^{0}\Vert^{2}_{L^{2}(-l,l)}),
\label{6}
\end{gather}
where $K=const>0$ does not depend on $\varepsilon$ and $\delta$. }}

\textit{\textrm{Using Galerkin's approximations, we get 
\begin{gather}
\beta(\varepsilon)\Vert
\partial_{t}u^{\delta}_{\varepsilon}\Vert^{2}_{L^{2}(l^{T}_{\varepsilon})}+%
\max \limits_{[0,T]}\Vert
\nabla_{x}u^{\delta}_{\varepsilon}\Vert^{2}_{L^{2}(\Omega)}+
\beta(\varepsilon)\delta^{-1}\max
\limits_{[0,T]}\Vert(u^{\delta}_{\varepsilon})^{-}\Vert^{2}_{L^{2}(l_{%
\varepsilon})}\le  \notag \\
\le K(\Vert{f}\Vert^{2}_{H^{1}(0,T;L^{2}(\Omega))}+\Vert{u}%
^{0}\Vert^{2}_{H^{1}_{0}(-l,l)}).  \label{7}
\end{gather}
}}

\textit{\textrm{From estimates \eqref{6}, \eqref{7}, we derive that there
exists a subsequence (we preserve the notation of the original sequence)
such that, as $\delta \to 0$, we have 
\begin{equation}  \label{8}
\begin{gathered} u_{\ee}^{\delta}\rightharpoonup u_{\ee}\, \, \mbox{weakly
in}\, \, L^{2}(0,T;H^{1}(\Omega,\Gamma_{1})),\\ \partial_{t}u^{\delta}_{\ee}\rightharpoonup \partial_{t}u_{\ee}\, \, \mbox{weakly in}\, \,L^{2}(0,T;L^{2}(l_{\ee})),\\ u^{\delta}_{\ee}\to u_{\ee}\, \, \mbox{strongly
in}\, \,C([0,T];L^{2}(l_{\ee})),\\ \Bigl(u_{\ee}^{\delta}\Bigr)^{-}\to 0\, \, \mbox{strongly in}\, \, L^{2}(0,T;L^{2}(l_{\ee})). \end{gathered}
\end{equation}
}}

\textit{\textrm{Let us prove that $u_{\varepsilon}\in \mathbb{K}%
_{\varepsilon}$ is a strong solution of the problem \eqref{1}. From the
integral identity \eqref{5}, we conclude 
\begin{gather}
\beta(\varepsilon)\int
\limits_{l^{T}_{\varepsilon}}\partial_{t}u^{\delta}_{\varepsilon}(\phi-u^{%
\delta}_{\varepsilon})dx_{1}dt+\int \limits_{Q^{T}}\nabla{u}%
^{\delta}_{\varepsilon}\nabla(\phi-u^{\delta}_{\varepsilon})dx{dt}+
\beta(\varepsilon)\int
\limits_{l^{T}_{\varepsilon}}\sigma(u^{\delta}_{\varepsilon})(\phi-u_{%
\varepsilon}^{\delta})dx_{1}dt+  \notag \\
+\beta(\varepsilon)\delta^{-1}\int \limits_{l^{T}_{\varepsilon}}\Bigl(%
u^{\delta}_{\varepsilon}\Bigr)^{-}(\phi-u^{\delta}_{\varepsilon})dx_{1}dt=%
\int \limits_{Q^{T}}f(\phi-u^{\delta}_{\varepsilon})dx_{1}dt,  \label{9}
\end{gather}
where $\phi$ is an arbitrary function from $\mathbb{K}_{\varepsilon}$.
Taking into account that $\phi \in \mathbb{K}_{\varepsilon}$, we have 
\begin{equation}  \label{10}
\beta(\varepsilon)\delta^{-1}\Bigl(\int
\limits_{l^{T}_{\varepsilon}}(u^{\delta}_{\varepsilon})^{-}\phi{d}%
x_{1}dt-\int
\limits_{l^{T}_{\varepsilon}}|(u^{\delta}_{\varepsilon})^{-}|^{2}dx_{1}dt%
\Bigr)\le 0,
\end{equation}
and, hence, we get 
\begin{gather}
\beta(\varepsilon)\int
\limits_{l^{T}_{\varepsilon}}\partial_{t}u^{\delta}_{\varepsilon}(\phi-u^{%
\delta}_{\varepsilon})dx_{1}dt+\int \limits_{Q^{T}}\nabla{u}%
_{\varepsilon}^{\delta}\nabla(\phi-u^{\delta}_{\varepsilon})dx{dt}+  \notag
\\
+\beta(\varepsilon)\int
\limits_{l^{T}_{\varepsilon}}\sigma(u^{\delta}_{\varepsilon})(\phi-u^{%
\delta}_{\varepsilon})dx_{1}dt\ge \int
\limits_{Q^{T}}f(\phi-u^{\delta}_{\varepsilon})dx{dt}.  \label{11}
\end{gather}
Using the estimate 
\begin{equation*}
\Vert{u_{\varepsilon}}\Vert_{L^{2}(0,T;H^{1}(\Omega,\Gamma_{1}))}\le \lim
\limits_{\delta \to 0}\Vert{u}^{\delta}_{\varepsilon}\Vert_{L^{2}(0,T;H^{1}(%
\Omega,\Gamma_{1}))},
\end{equation*}
and applying 
\begin{gather*}
\beta(\varepsilon)\int
\limits_{l^{T}_{\varepsilon}}(\sigma(u^{\delta}_{\varepsilon})-\sigma(u_{%
\varepsilon}))(\phi-u^{\delta}_{\varepsilon})dx_{1}dt= \\
=\beta(\varepsilon)\int
\limits_{l^{T}_{\varepsilon}}(\sigma(u^{\delta}_{\varepsilon})-\sigma(u_{%
\varepsilon}))(\phi-u_{\varepsilon})dx_{1}dt- \beta(\varepsilon)\int
\limits_{l^{T}_{\varepsilon}}(\sigma(u^{\delta}_\varepsilon)-\sigma(u_{%
\varepsilon}))(u^{\delta}_{\varepsilon}-u_{\varepsilon})dx_{1}dt\le \\
\le \beta(\varepsilon)\int
\limits_{l^{T}_{\varepsilon}}(\sigma(u^{\delta}_{\varepsilon})-\sigma(u_{%
\varepsilon}))(\phi-u_{\varepsilon})dx_{1}dt\to 0,\, \delta \to 0,
\end{gather*}
we conclude that $u^{\delta}_{\varepsilon}$ converges to the strong solution 
$u_{\varepsilon}\in \mathbb{K}_{\varepsilon}$ of the problem \eqref{1}. From
the estimates \eqref{6}, \eqref{7}, we obtain the estimates \eqref{3}. The
uniqueness of the strong solution of the problem \eqref{1} immediately
follows from the inequality \eqref{2}. }}
\end{proof}

\textit{\textrm{From the estimates \eqref{3}, we have as $\varepsilon \to 0$ 
\begin{equation}  \label{12}
u_{\varepsilon}\rightharpoonup u_{0}\, \, \mbox{weakly in}\,
\,L^{2}(0,T;H^{1}(\Omega,\Gamma_{1})).
\end{equation}
}}

\subsection{Main theorem for the initial problem}

\begin{theorem}
\label{main hom theorem} \textit{\textrm{Let $u_{\varepsilon}$ be a strong
solution of the problem \eqref{1}. Then the function $u_{0}$, defined by %
\eqref{12} is a weak solution of the system 
\begin{equation}  \label{13}
\left \{%
\begin{array}{lr}
-\Delta_{x}{u}_{0}(x, t) = f(x, t), & (x,t)\in Q^{T}, \\ 
-\partial_{x_{2}}u_{0} + \mathcal{M}u_{0} = \mathcal{M}H_{u_{0}}, & (x,
t)\in \Gamma^{T}_{2}, \\ 
H_{u_{0}}\ge 0,\, \partial_{t}H_{u_{0}}+\mathcal{L}H_{u_{0}}+%
\sigma(H_{u_{0}})\ge \mathcal{L}u_{0}, & (x,t)\in \Gamma^{T}_{2}, \\ 
H_{u_{0}}(\partial_{t}H_{u_{0}}+\mathcal{L}H_{u_{0}}+\sigma(H_{u_{0}})-%
\mathcal{L}u_{0})=0, & (x,t)\in \Gamma^{T}_{2}, \\ 
H_{u_{0}}(x,0)=u^{0}(x), & x\in \Gamma_{2}, \\ 
u_{0}(x,t)=0, & (x,t)\in \Gamma^{T}_{1},%
\end{array}%
\right.
\end{equation}
where $\mathcal{M}=\frac{\pi}{\alpha^{2}}$, $\mathcal{L}=\frac{\pi}{%
2C_{0}l_{0}\alpha^{2}}$. }}
\end{theorem}

\begin{remark}
\textit{\textrm{In addition, if $u\in H^{1}(0,T;L^{2}(\Omega))\cap
L^{2}(0,T;H^{1}_{0}(\Omega))$, then, $u(x,0)$ is given as the unique
solution of the stationary problem 
\begin{equation*}
\left \{%
\begin{array}{lr}
-\Delta{u}(x,0)=f(x,0), & x\in \Omega, \\ 
u(x,0)=0, & x\in \Gamma_{1}, \\ 
\partial_{x_{2}}u(x,0)=\mathcal{M}(u(x,0)-u^{0}(x)), & x\in \Gamma_{2}.%
\end{array}%
\right.
\end{equation*}
}}
\end{remark}

\begin{remark}
\textit{\textrm{As we can see, a new ``strange term'' has appeared in the
boundary condition of the homogenized problem. This new term is a non-local
nonlinear operator that requires solving an obstacle problem for the
ordinary differential operator to determine it 
\begin{equation}  \label{14}
\left \{%
\begin{array}{lr}
\frac{d}{dt}H_{\phi}+\mathcal{L}H_{\phi} + \sigma(H_{\phi})\ge \mathcal{L}%
\phi, &  \\ 
H_{\phi}\ge 0, & t\in (0,T), \\ 
H_{\phi}(\frac{d}{dt}H_{\phi}+\mathcal{L}H_{\phi} + \sigma(H_{\phi})- 
\mathcal{L}\phi)=0, &  \\ 
H_{\phi}(0)=u^{0}. & 
\end{array}%
\right.
\end{equation}
}}

\textit{\textrm{The function $H_{\phi}$ is a weak solution of the problem %
\eqref{14}, if $H_{\phi}\in H^{1}(0,T)$, $H_{\phi}(0)=u^{0}\ge 0,$ $%
H_{\phi}(t)\ge 0$ for any $t\in [0,T]$ and for an arbitrary function $\psi
\in L^{2}(0,T)$, $\psi(t)\ge 0$, $t\in (0,T)$, it satisfies the variational
inequality 
\begin{gather}
\int \limits_{0}^{T}\frac{d}{dt}H_{\phi}(\psi-H_{\phi})dt+\mathcal{L}\int
\limits_{0}^{T}H_{\phi}(\psi-H_{\phi})dt+\int
\limits_{0}^{T}\sigma(H_{\phi})(\psi-H_{\phi})dt\ge  \notag \\
\ge \mathcal{L}\int \limits_{0}^{T}\phi(\psi-H_{\phi})dt.  \label{15}
\end{gather}
}}

\textit{\textrm{It is easy to show that problem \eqref{14} has the unique
weak solution. Indeed, if we denote by $H_{1,\phi}$ and $H_{2,\phi}$ two
weak solutions of the problem \eqref{14}, then taking $\psi=\frac{1}{2}%
(H_{1,\phi} + H_{2,\phi})$ as a test-function in \eqref{15}, we get 
\begin{gather}
\int \limits_{0}^{T}\frac{d}{dt}(H_{2,\phi}-H_{1,\phi})(H_{1,\phi}-H_{2,%
\phi})dt + \mathcal{L}\int
\limits_{0}^{T}(H_{2,\phi}-H_{1,\phi})(H_{1,\phi}-H_{2,\phi})dt +  \notag \\
+\int
\limits_{0}^{T}(\sigma(H_{2,\phi})-\sigma(H_{1,\phi}))(H_{1,\phi}-H_{2,%
\phi})dt\ge 0.  \label{16}
\end{gather}
Evidently, the left-hand side has a non-positive value. So, we conclude that 
$H_{1,\phi}=H_{2,\phi}$. To prove the existence of a weak solution of the
problem \eqref{14}, we use the penalized method. Consider the problem 
\begin{equation}  \label{17}
\left \{%
\begin{array}{lr}
\frac{d}{dt}H_{\phi,\delta}+\mathcal{L}H_{\phi,\delta}+\sigma(H_{\phi,%
\delta}) + \delta^{-1}H_{\phi,\delta}^{-}=\mathcal{L}\phi, & t\in (0,T), \\ 
H_{\phi,\delta}(0)=u^{0}, & 
\end{array}%
\right.
\end{equation}
where $\delta>0$, $g^{+}(t)=\sup \{0, g(t)\}$, $g^{-} = g - g^{+}$. By the
solution of the problem \eqref{15}, we consider a function $%
H_{\phi,\delta}\in H^{1}(0,T)$, $H_{\phi,\delta}(0)=u^{0}$, that satisfies
the following identity 
\begin{equation}  \label{18}
\int \limits_{o}^{T}\frac{d}{dt}H_{\phi,\delta}\psi{dt}+\mathcal{L}\int
\limits_{0}^{T}H_{\phi,\delta}\psi{dt}+ \int
\limits_{0}^{T}\sigma(H_{\phi,\delta})\psi{dt}+ \delta^{-1}\int
\limits_{0}^{T}H_{\phi,\delta}^{-}\psi{dt}=\mathcal{L}\int
\limits_{0}^{T}\phi \psi{dt},
\end{equation}
for an arbitrary function $\psi \in L^{2}(0, T)$. }}

\textit{\textrm{Taking $\psi=H_{\phi,\delta}$ as a test function in %
\eqref{18}, we obtain 
\begin{equation}  \label{19}
\max \limits_{[0,T]}H^{2}_{\phi,\delta}+\Vert{H_{\phi,\delta}}%
\Vert^{2}_{L^{2}(0,T)}+\delta^{-1}\Vert{H_{\phi,\delta}}^{-}%
\Vert^{2}_{L^{2}(0,T)}\le C(\Vert \phi \Vert^{2}_{L^{2}(0,T)}+(u^{0})^{2}),
\end{equation}
From the equation of the problem \eqref{17} and the estimate \eqref{19}, we
conclude 
\begin{equation}  \label{20}
\Vert \partial_{t}H_{\phi,\delta}\Vert^{2}_{L^{2}(0,T)}\le K(\Vert \phi
\Vert^{2}_{L^{2}(0,T)}+(u^{0})^{2}),
\end{equation}
where $K$ and $C$ do not depend on $\varepsilon$ and $\delta$. Indeed, we
have }}

\textit{\textrm{%
\begin{equation*}
\int \limits_{0}^{T}(\frac{d}{dt}H_{\phi,\delta})^{2}dt+\mathcal{L}^{2}\int
\limits_{0}^{T}H^{2}_{\phi,\delta}dt+\int
\limits_{0}^{T}\sigma^{2}(H_{\phi,\delta})dt+ \delta^{-2}\int
\limits_{0}^{T}(H^{-}_{\phi,\delta})^{2}dt+\newline
\end{equation*}
\begin{equation*}
+2\mathcal{L}\int \limits_{0}^{T}H^{^{\prime
}}_{\phi,\delta}H_{\phi,\delta}dt+2\int \limits_{0}^{T}H^{^{\prime
}}_{\phi,\delta}\sigma(H_{\phi,\delta})dt+ 2\delta^{-1}\int
\limits_{0}^{T}H^{^{\prime }}_{\phi,\delta}H^{-}_{\phi,\delta}dt+ 2\mathcal{L%
}\delta^{-1}\int \limits_{0}^{T}H_{\phi,\delta}H^{-}_{\phi \delta}dt+\newline
\end{equation*}
\begin{equation*}
+2\mathcal{L}\int
\limits_{0}^{T}H_{\phi,\delta}\sigma(H_{\phi,\delta})dt+2\delta^{-1}\int
\limits_{0}^{T}\sigma(H_{\phi,\delta})H^{-}_{ \phi,\delta}dt = \mathcal{L}%
^{2}\int \limits_{0}^{T}\phi^{2}dt.\newline
\end{equation*}
From here, taking into account that 
\begin{gather*}
2\delta^{-1}\int \limits_{0}^{T}H^{^{\prime
}}_{\phi,\delta}H^{-}_{\phi,\delta}dt=\delta^{-1}(H^{-}_{\phi,%
\delta}(T))^{2}\ge 0, \\
2\delta^{-1}\int
\limits_{0}^{T}\sigma(H_{\phi,\delta})H^{-}_{\phi,\delta}=2\delta^{-1}\int%
\limits^T_0\sigma(H^{-}_{\phi,\delta})H^{-}_{\phi,\delta}{dt}\ge 0,\newline
\end{gather*}
and using Cauchy inequality $ab\le \alpha{a}^{2}+C_{\alpha}b^{2}$ for any $%
\alpha>0$ and \eqref{19}, we get \eqref{20}. }}

\textit{\textrm{Then, using the estimates \eqref{19} and \eqref{20}, we
derive, as $\delta \to 0$, convergences 
\begin{gather}
H_{\phi,\delta}\rightharpoonup H_{\phi}\, \, \mbox{weakly in }%
\,H^{1}(0,T),\, \,H_{\phi,\delta}\to H_{\phi}\, \, 
\mbox{ uniformly with
respect to }\,t\in [0,T],  \notag \\
H_{\phi,\delta}^{-}\to 0=H^{-}_{\phi}\, \, \mbox{in}\, \, L^{2}(0,T),\, \, 
\frac{d}{dt}H_{\phi,\delta}\rightharpoonup \frac{d}{dt}H_{\phi}\, \, %
\mbox{weakly in}\, \,L^{2}(0,T).  \label{21}
\end{gather}
From \eqref{19}--\eqref{21}, we obtain the estimates 
\begin{equation}  \label{22}
\max \limits_{[0,T]}|H_{\phi}|\le K(\Vert \phi
\Vert_{L^{2}(0,T)}+u^{0}),\quad \Vert{H}_{\phi}\Vert_{H^{1}(0,T)}\le K(\Vert
\phi \Vert_{L^{2}(0,T)}+u^{0}).
\end{equation}
Using $\psi=v - H_{\phi,\delta}$, where $v\ge 0$ for $t\in [0,T]$, we get 
\begin{gather}
\int \limits_{0}^{T}\frac{d}{dt}H_{\phi,\delta}(v-H_{\phi,\delta})dt + 
\mathcal{L}\int \limits_{0}^{T}H_{\phi,\delta}(v-H_{\phi,\delta})dt + \int
\limits_{0}^{T}\sigma(H_{\phi,\delta})(v-H_{\phi,\delta})dt +  \notag \\
+\delta^{-1}\int \limits_{0}^{T}H_{\phi,\delta}^{-}(v-H_{\phi,\delta})dt=%
\mathcal{L}\int \limits_{0}^{T}\phi(v-H_{\phi,\delta})dt.  \label{23}
\end{gather}
}}

\textit{\textrm{Taking into account that 
\begin{equation*}
\delta^{-1}\int \limits_{0}^{T}H^{-}_{\phi,\delta}(v-H_{\phi,\delta})dt\le 0,
\end{equation*}
and using convergences \eqref{21}, we derive the inequality \eqref{15}. }}

\textit{\textrm{We denote by $\mathbf{H}$ the operator from $%
L^{2}(0,T)\times \mathbb{R}^{+}$ to $L^{2}(0,T)$ that maps a function $\phi$
and nonnegative number $u^{0}$ to a solution $H_{\phi,u^{0}}(t)$ of the
problem \eqref{14}. Then, we have }}

\begin{theorem}
\textit{\textrm{The operator $\mathbf{H}$, defined by the equality $\mathbf{H%
}(\phi,u^{0})=H_{\phi,u^{0}}(t)$ satisfies the following inequalities: let $%
\phi_{1},\, \phi_{2} \in L^{2}(0,T)$, and $u^{0}_{1},u^{0}_{2}\ge 0$, then 
\begin{equation}  \label{24}
\max \limits_{[0,T]}|H_{\phi_{1}, u^{0}_{1}}-H_{\phi_{2},u^{0}_{2}}|\le
K(|u^{0}_{1}-u^{0}_{2}|+\Vert \phi_{1}-\phi_{2}\Vert_{L^{2}(0,T)}),
\end{equation}
\begin{gather}
\mathcal{L}\int
\limits_{0}^{T}(\phi_{1}-\phi_{2})(H_{\phi_{1},u^{0}_{1}}-H_{%
\phi_{2},u^{0}_{2}})dt+\frac{1}{2}(u^{0}_{1}-u^{0}_{2})^{2}\ge  \notag \\
\ge \frac{1}{2}(H_{\phi_{1},u^{0}_{1}}(T)-H_{\phi_{2},u^{0}_{2}}(T))^{2}+%
\mathcal{L}\int
\limits_{0}^{T}(H_{\phi_{1},u^{0}_{1}}-H_{\phi_{2},u^{0}_{2}})^{2}dt.
\label{24_2}
\end{gather}
}}
\end{theorem}

\begin{proof}
\textit{\textrm{We take $v=\frac{1}{2}(H_{\phi_{1},u^{0}_{1}} +
H_{\phi_{2},u^{0}_{2}})$ as a test function in variational inequalities for
functions $H_{\phi_{1},u^{0}_{1}}(t)$, $H_{\phi_{2},u^{0}_{2}}(t)$. Then, we
sum the obtained expressions and get 
\begin{gather*}
\int \limits_{0}^{T}\frac{d}{dt}(H_{\phi_{1},u^{0}_{1}}-H_{%
\phi_{2},u^{0}_{2}})(H_{\phi_{1},u^{0}_{1}}-H_{\phi_{2},u^{0}_{2}})dt+%
\mathcal{L}\int
\limits_{0}^{T}(H_{\phi_{1},u^{0}_{1}}-H_{\phi_{2},u^{0}_{2}})^{2}dt+  \notag
\\
+\int
\limits_{0}^{T}(\sigma(H_{\phi_{1},u^{0}_{1}})-\sigma(H_{%
\phi_{2},u^{0}_{2}}))(H_{\phi_{1},u^{0}_{1}}-H_{\phi_{2},u^{0}_{2}})dt\le 
\mathcal{L}\int
\limits_{0}^{T}(\phi_{1}-\phi_{2})(H_{\phi_{1},u^{0}_{1}}-H_{%
\phi_{2},u^{0}_{2}})dt.
\end{gather*}
From these inequalities, we deduce \eqref{24}, \eqref{24_2}. }}
\end{proof}
\end{remark}

\subsection{\textit{\textrm{Proof of the Main Theorem for the initial problem%
}}}

\textit{\textrm{We define a set of obstacle problems 
\begin{equation}  \label{25}
\left \{%
\begin{array}{lr}
\frac{d}{dt}H^{j}_{\phi,\varepsilon}+\mathcal{L}H^{j}_{\phi,\varepsilon}+%
\sigma(H^{j}_{\phi,\varepsilon})\ge \mathcal{L}\phi(P^{j}_{\varepsilon},t),
\, \,H^{j}_{\phi,\varepsilon}(t)\ge 0, & t\in (0,T), \\ 
(\frac{d}{dt}H^{j}_{\phi,\varepsilon}+\mathcal{L}H^{j}_{\phi,\varepsilon}+%
\sigma(H^{j}_{\phi,\varepsilon})-\mathcal{L}\phi(P^{j}_{%
\varepsilon},t))H^{j}_{\phi,\varepsilon}=0,\, & t\in (0,T), \\ 
H^{j}_{\phi,\varepsilon}(0)=u^{0}(P^{j}_{\varepsilon}), & 
\end{array}%
\right.
\end{equation}
where $j\in \Upsilon_{\varepsilon}$, $P^{j}_{\varepsilon}={\varepsilon}j$, $%
\phi(x,t)=\psi(x)\eta(t)$, $\psi \in C^{\infty}(\overline \Omega,\Gamma_{1})$%
, $\eta \in C^{1}[0,T]$. }}

\textit{\textrm{We denote by $T^{j}_{r}$ the ball of radius $r$ with the
center in $P^{j}_{\varepsilon}={\varepsilon}j$ and $%
(T^{j}_{r})^{+}=T^{j}_{r}\cap \{x_{2}>0\}$. Consider auxiliary functions $%
w^{j}_{\varepsilon}$ and $q^{j}_{\varepsilon}$ which are solutions of the
problems 
\begin{equation}  \label{26}
\Delta{w}^{j}_{\varepsilon}=0,\,x\in T^{j}_{\varepsilon/4}\setminus 
\overline{T^{j}_{a_{\varepsilon}}},\, \,w^{j}_{\varepsilon}=1,\,x\in \partial%
{T}^{j}_{a_{\varepsilon}},\,w^{j}_{\varepsilon}=0, \,x\in \partial{T}%
^{j}_{\varepsilon/4},
\end{equation}
and 
\begin{equation}  \label{27}
\Delta{q}^{j}_{\varepsilon}=0,\, \,x\in T^{j}_{{\varepsilon}/4}\setminus%
\overline{l^{j}_{\varepsilon}},\, \,q^{j}_{\varepsilon}=1, \,x\in
l^{j}_{\varepsilon},\, \,q^{j}_{\varepsilon}=0,\, \,x\in \partial{%
T^{j}_{\varepsilon/4}}.
\end{equation}
Note that $w^{j}_{\varepsilon}$ and $q^{j}_{\varepsilon}$ are solutions of
the problems 
\begin{equation}  \label{28}
\left \{%
\begin{array}{lr}
\Delta{w}^{j}_{\varepsilon}=0,\, & x\in (T^{j}_{\varepsilon/4})^{+}\setminus%
\overline{T^{j}_{a_{\varepsilon}}}, \\ 
w^{j}_{\varepsilon}=0,\, & x\in (\partial{T^{j}_{\varepsilon/4}})^{+}, \\ 
w^{j}_{\varepsilon}=1,\, & x\in (\partial{T^{j}_{a_{\varepsilon}}})^{+}, \\ 
\partial_{x_{2}}w^{j}_{\varepsilon}=0, & x\in \{x_{2}=0\}
\cap(T^{j}_{\varepsilon/4}\setminus \overline{T^{j}_{a_{\varepsilon}}}),%
\end{array}%
\right.
\end{equation}
and 
\begin{equation}  \label{29}
\left \{%
\begin{array}{lr}
\Delta{q}^{j}_{\varepsilon}=0, & x\in (T^{j}_{\varepsilon/4})^{+}, \\ 
q^{j}_{\varepsilon}=0, & x\in (\partial{T}^{j}_{\varepsilon/4})^{+}, \\ 
q^{j}_{\varepsilon}=1, & x\in l^{j}_{\varepsilon}, \\ 
\partial_{x_{2}}q^{j}_{\varepsilon}=0, & x\in (T^{j}_{\varepsilon/4}\cap
\{x_{2}=0\})\setminus \overline{l^{j}_{\varepsilon}},%
\end{array}%
\right.
\end{equation}
where $j\in \Upsilon_{\varepsilon}$. It is easy to see that 
\begin{equation*}
w^{j}_{\varepsilon}=\frac{\ln \Bigl(\frac{4r}{\varepsilon}\Bigr)}{\ln \Bigl(%
\frac{4a_{\varepsilon}}{\varepsilon}\Bigr)}.
\end{equation*}
}}

\textit{\textrm{We define the functions $W_{\varepsilon}(x)$ and $%
Q_{\varepsilon}(x)$ by setting 
\begin{equation}  \label{30}
W_{\varepsilon}(x) = \left \{%
\begin{array}{lr}
w^{j}_{\varepsilon}(x), & x\in \Bigl(T^{j}_{\varepsilon/4}\setminus 
\overline{T^{j}_{a_{\varepsilon}}}\Bigr)^{+},\,j\in \Upsilon_{\varepsilon},
\\ 
1, & x\in \overline{(T^{j}_{a_{\varepsilon}})^{+}},\,j\in
\Upsilon_{\varepsilon}, \\ 
0, & x\in \Omega \setminus \overline{\cup_{j\in
\Upsilon_{\varepsilon}}(T^{j}_{\varepsilon/4})^{+}},%
\end{array}%
\right.
\end{equation}
and 
\begin{equation}  \label{31}
Q_{\varepsilon}(x)=\left \{%
\begin{array}{lr}
q^{j}_{\varepsilon}(x), & x\in \Bigl(T^{j}_{\varepsilon/4}\Bigr)^{+},\,j\in
\Upsilon_{\varepsilon}, \\ 
0, & x\in \Omega \setminus{\overline{\cup_{j\in
\Upsilon_{\varepsilon}}(T^{j}_{\varepsilon/4})^{+}}}.%
\end{array}%
\right.
\end{equation}
It is easy to see that $W_{\varepsilon},\, Q_{\varepsilon}\in H^{1}(\Omega,
\Gamma_{1})$ and $W_{\varepsilon}\rightharpoonup 0$ weakly in $%
H^{1}(\Omega,\Gamma_{1})$ as $\varepsilon \to 0$. To compare these two
functions, we will use the following lemma, proved in \cite{D3GCPodSh2018}. }%
}

\begin{lemma}
\textit{\textrm{Let $W_{\varepsilon}$ and $Q_{\varepsilon}$ are the
functions defined by \eqref{30} and \eqref{31} correspondingly. Then, we
have the estimate 
\begin{equation}  \label{32}
\Vert{W_{\varepsilon}-Q_{\varepsilon}}\Vert_{H^{1}(\Omega,\Gamma_{1})}\le K%
\sqrt{\varepsilon}.
\end{equation}
}}
\end{lemma}

\textit{\textrm{Let's construct two more auxiliary functions 
\begin{equation}  \label{33}
Q_{\varepsilon,\phi}(x,t)=\left \{%
\begin{array}{lr}
q^{j}_{\varepsilon}(x)(\phi(x,t)-H^{j}_{\phi,\varepsilon}(t)), & x\in(
T^{j}_{\varepsilon/4})^{+},\,j\in \Upsilon_{\varepsilon},\,t\in (0,T), \\ 
0, & x\in \Omega \setminus \overline{\cup_{j\in
\Upsilon_{\varepsilon}}T^{j}_{\varepsilon/4}},%
\end{array}%
\right.
\end{equation}
and 
\begin{equation}  \label{34}
W_{\varepsilon,\phi}(x,t)=\left \{%
\begin{array}{lr}
w^{j}_{\varepsilon}(x)(\phi(x,t)-H^{j}_{\phi,\varepsilon}(t)), & x\in \Bigl(%
T^{j}_{\varepsilon/4}\Bigr)^{+}\setminus \overline{T^{j}_{a_{\varepsilon}}}%
,\,j\in \Upsilon_{\varepsilon},\,t\in (0,T), \\ 
\varphi(x, t) - H^j_{\phi, \varepsilon}(t), & x\in \overline{(T^{j}_{a_{\varepsilon}})^{+}}, \,j\in
\Upsilon_{\varepsilon},\,t\in (0,T), \\ 
0, & x\in \Omega \setminus \overline{\bigcup \limits_{j\in
\Upsilon_{\varepsilon}}T^{j}_{\varepsilon/4}},\,t\in (0,T).%
\end{array}%
\right.
\end{equation}
Using the properties of $W_{\varepsilon}$, we have that $W_{\varepsilon,%
\phi}\rightharpoonup 0$ weakly in $L^{2}(0,T; H^{1}(\Omega,\Gamma_{1}))$ as $%
\varepsilon \to 0$. }}

\textit{\textrm{From Lemma~1, we deduce }}

\begin{lemma}
\textit{\textrm{The following estimate is valid 
\begin{equation}  \label{35}
\Vert{W_{\varepsilon,\phi}-Q_{\varepsilon,\phi}}\Vert_{L^{2}(0,T;
H^{1}(\Omega,\Gamma_{1}))}\le K\sqrt{\varepsilon}.
\end{equation}
}}
\end{lemma}

\textit{\textrm{From the variational inequality \eqref{2}, using the
monotonicity of $\sigma$, we have 
\begin{gather}
\beta(\varepsilon)\int
\limits_{l^{T}_{\varepsilon}}\partial_{t}\psi(\psi-u_{\varepsilon})dx_{1}dt+%
\int \limits_{Q^{T}}\nabla \psi \nabla(\psi-u_{\varepsilon})dx{dt}+
\beta(\varepsilon)\int
\limits_{l^{T}_{\varepsilon}}\sigma(\psi)(\psi-u_{\varepsilon})dx_{1}dt\ge 
\notag \\
\ge \int \limits_{Q^{T}}f(\psi-u_{\varepsilon})dx{dt}-\frac{%
\beta(\varepsilon)}{2}\Vert
\psi(x,0)-u^{0}(x)\Vert^{2}_{L^{2}(l_{\varepsilon})},  \label{36}
\end{gather}
where $\psi \in \mathbb{K}_{\varepsilon}$. }}

\textit{\textrm{Let us take in \eqref{36} as a test function $\psi=\phi(x,t)
- Q_{\varepsilon,\phi}(x,t)$, where $\phi$ is an arbitrary function from $%
L^{2}(0,T; H^{1}(\Omega,\Gamma_{1}))$. Note, that $(\phi(x,t) -
Q_{\varepsilon,\phi}(x,t))\Bigl|_{x\in l^{j}_{\varepsilon},t\in (0,T)} =
\phi(x,t) -\phi(x,t)+H^{j}_{\phi,\varepsilon}=H^{j}_{\phi,\varepsilon}(t)\ge
0$, hence, $\psi \in \mathbb{K}_\varepsilon$. Also, we have $%
\psi(x,0)-u^{0}(x)=\phi(x,0)-\phi(x,0)+H^{j}_{\phi,%
\varepsilon}(0)-u^{0}(x)=u^{0}(P^{j}_{\varepsilon})-u^{0}(x)$, $x\in
l^{j}_{\varepsilon}$. Thus, from the inequality \eqref{36}, we derive 
\begin{gather}
\beta(\varepsilon)\int
\limits_{l^{T}_{\varepsilon}}\partial_{t}(\phi-Q_{\varepsilon,\phi})(%
\phi-Q_{\varepsilon,\phi}-u_{\varepsilon})dx_{1}dt+\int
\limits_{Q^{T}}\nabla(\phi-Q_{\varepsilon,\phi})\nabla(\phi-Q_{\varepsilon,%
\phi}-u_{\varepsilon})dx{dt} +  \notag \\
+\beta(\varepsilon)\int
\limits_{l^{T}_{\varepsilon}}\sigma(\phi-Q_{\varepsilon,\phi})(\phi-Q_{%
\varepsilon,\phi}-u_{\varepsilon})dx_{1}dt\ge \int
\limits_{Q^{T}}f(\phi-Q_{\varepsilon,\phi}-u_{\varepsilon})dx{dt}%
+\alpha_{\varepsilon},  \label{37}
\end{gather}
where $\alpha_{\varepsilon}\to 0$ as $\varepsilon \to 0$. Again, using that $%
(\phi(x,t) - Q_{\varepsilon,\phi}(x,t))\Bigl|_{x\in l^{j}_{\varepsilon},t\in
(0,T)} = H^j_{\phi, \varepsilon}$, from \eqref{37}, we conclude 
\begin{gather}
\beta(\varepsilon)\sum \limits_{j\in \Upsilon_{\varepsilon}}\int
\limits_{0}^{T}\int
\limits_{l^{j}_{\varepsilon}}\partial_{t}H^{j}_{\varepsilon,\phi}(H^{j}_{%
\varepsilon,\phi}-u_{\varepsilon})dx_{1}dt+ \int \limits_{Q^{T}}\nabla \phi
\nabla(\phi-Q_{\varepsilon,\phi}-u_{\varepsilon})dx{dt} -  \notag \\
-\int \limits_{Q^{T}}\nabla{Q_{\varepsilon,\phi}}\nabla(\phi-Q_{\varepsilon,%
\phi}-u_{\varepsilon})dx{dt}+ \beta(\varepsilon)\sum \limits_{j\in
\Upsilon_{\varepsilon}}\int \limits_{0}^{T}\int
\limits_{l^{j}_{\varepsilon}}\sigma(H^{j}_{\varepsilon,\phi})(H^{j}_{%
\varepsilon,\phi}-u_{\varepsilon})dx_{1}dt\ge  \notag \\
\ge \int \limits_{Q^{T}}f(\phi-Q_{\varepsilon,\phi} - u_{\varepsilon})dx{dt}%
+\alpha{\varepsilon}.  \label{38}
\end{gather}
}}

\textit{\textrm{Using Lemma~2, we have 
\begin{equation}  \label{39}
\lim \limits_{\varepsilon \to 0}\int
\limits_{Q^{T}}\nabla(Q_{\varepsilon,\phi}-W_{\varepsilon,\phi})\nabla(%
\phi-Q_{\varepsilon,\phi}-u_{\varepsilon})dx{dt}=0.
\end{equation}
}}

\textit{\textrm{Let us consider the following integral 
\begin{equation*}
J_{\varepsilon}\equiv \int \limits_{Q^{T}}\nabla{W_{\varepsilon,\phi}}%
\nabla(\phi-Q_{\varepsilon,\phi}-u_{\varepsilon})dx{dt}.
\end{equation*}
}}

\textit{\textrm{We have 
\begin{gather}
J_{\varepsilon}=\sum \limits_{j\in \Upsilon_{\varepsilon}}\int
\limits_{0}^{T}\int \limits_{(T^{j}_{\varepsilon/4})^{+}\setminus \overline{%
(T^{j}_{a_{\varepsilon}})^{+}}}\nabla({w}^{j}_{\varepsilon}(%
\phi(x,t)-H^{j}_{\varepsilon}(t)))\nabla(\phi-Q_{\varepsilon,\phi}-u_{%
\varepsilon})dx{dt} =  \notag \\
=\sum \limits_{j\in \Upsilon_{\varepsilon}}\int \limits_{0}^{T}\int
\limits_{(T^{j}_{\varepsilon/4})^{+}\setminus \overline{(T^{j}_{a_{%
\varepsilon}})^{+}}}\nabla{w}^{j}_{\varepsilon}\nabla((\phi-H^{j}_{%
\varepsilon})(\phi-Q_{\varepsilon,\phi}-u_{\varepsilon}))dx{dt}%
+\alpha_{\varepsilon}=  \notag \\
=\sum \limits_{j\in \Upsilon_{\varepsilon}}\int \limits_{0}^{T}\int
\limits_{(\partial{T}^{j}_{\varepsilon/4})^{+}}\partial_{\nu}w^{j}_{%
\varepsilon}(\phi(x,t)-H^{j}_{\varepsilon}(t))(\phi-Q_{\varepsilon,\phi}-u_{%
\varepsilon})ds{dt} +  \notag \\
+\sum \limits_{j\in \Upsilon_{\varepsilon}}\int \limits_{(\partial{T}%
^{j}_{a_{\varepsilon}})^{+}}\partial_{\nu}w^{j}_{\varepsilon}(%
\phi(x,t)-H^{j}_{\varepsilon})(\phi-Q_{\varepsilon,\phi}-u_{\varepsilon})ds{%
dt}+\alpha_{\varepsilon}.  \label{40}
\end{gather}
}}

\textit{\textrm{Taking into account that $\partial_{\nu}w^{j}_{\varepsilon}%
\Bigl|_{\partial{T^{j}_{\varepsilon/4}}}=\frac{4}{-\alpha^{2}+{\varepsilon}%
\ln(4C_{0})}$ and $\partial_{\nu}w^{j}_{\varepsilon}\Bigl|_{\partial{%
T^{j}_{a_{\varepsilon}}}}= \frac{\exp(\alpha^{2}/\varepsilon)}{%
C_{0}\alpha^{2}-C_{0}{\varepsilon}\ln(4C_{0})}$, we transform relation %
\eqref{40} into 
\begin{gather}
J_{\varepsilon}=-\frac{4}{\alpha^{2}}\sum \limits_{j\in
\Upsilon_{\varepsilon}}\int \limits_{0}^{T}\int \limits_{(\partial{T}%
^{j}_{\varepsilon/4})^{+}}(\phi(x,t)-H^{j}_{\varepsilon}(t))(\phi-u_{%
\varepsilon})ds{dt} +  \notag \\
+\frac{\beta(\varepsilon)}{C_{0}\alpha^{2}}\sum \limits_{j\in
\Upsilon_{\varepsilon}}\int \limits_{(\partial{T^{j}_{a_{\varepsilon}}}%
)^{+}}(\phi(x,t)-H^{j}_{\varepsilon,\phi}(t))(\phi - Q_{\varepsilon,\phi} -
u_{\varepsilon})ds{dt} + \alpha_{\varepsilon},  \label{41}
\end{gather}
where $\alpha_{\varepsilon}\to 0$ as $\varepsilon \to 0$. }}

\textit{\textrm{In what follows, we will need the following lemma of \cite%
{D3GCPodSh2018}. }}

\begin{lemma}
\textit{\textrm{Let $h\in H^{1}(\Omega, \Gamma_{1})$. Then 
\begin{equation}  \label{42}
\Bigl|\frac{\beta(\varepsilon)\pi}{2l_{0}}\int \limits_{l_{\varepsilon}}h{%
dx_{1}}-\beta(\varepsilon)\sum \limits_{j\in \Upsilon_{\varepsilon}}\int
\limits_{(\partial{T^{j}_{a_{\varepsilon}}})^{+}}h{ds}\Bigr|\le K\sqrt{%
\varepsilon}\Vert{h}\Vert_{H^{1}(\Omega,\Gamma_{1})}.
\end{equation}
}}
\end{lemma}

\textit{\textrm{By virtue of Lemma~3, we have 
\begin{gather}
\frac{1}{C_{0}\alpha^{2}}\Bigl|\beta(\varepsilon)\sum \limits_{j\in
\Upsilon_{\varepsilon}}\int \limits_{0}^{T}\int \limits_{(\partial{%
T^{j}_{a_{\varepsilon}}})^{+}}(\phi(x,t)-H^{j}_{\varepsilon,\phi}(t))(%
\phi-Q_{\varepsilon,\phi}-u_{\varepsilon})ds{dt} -  \notag \\
-\frac{\pi \beta(\varepsilon)}{2l_{0}}\sum \limits_{j\in
\Upsilon_{\varepsilon}}\int \limits_{0}^{T}\int
\limits_{l^{j}_{\varepsilon}}(\phi(x,t)-H^{j}_{\varepsilon,\phi}(t))(H^{j}_{%
\varepsilon,\phi}-u_{\varepsilon})dx_{1}dt\Bigr|\le K\sqrt{\varepsilon}.
\label{43}
\end{gather}
}}

\textit{\textrm{From \eqref{41} and \eqref{43}, we obtain 
\begin{gather}
J_{\varepsilon}=-\frac{4}{\alpha^{2}}\sum \limits_{j\in
\Upsilon_{\varepsilon}}\int \limits_{0}^{T}\int \limits_{(\partial{%
T^{j}_{\varepsilon/4}})^{+}}(\phi(x,t)-H^{j}_{\varepsilon,\phi}(t))(\phi-u_{%
\varepsilon})ds{dt} +  \notag \\
+\frac{\pi \beta(\varepsilon)}{2\alpha^{2}C_{0}l_{0}}\sum \limits_{j\in
\Upsilon_{\varepsilon}}\int \limits_{0}^{T}\int
\limits_{l^{j}_{\varepsilon}}(\phi(x,t)-H^{j}_{\varepsilon,\phi})(H^{j}_{%
\varepsilon,\phi}-u_{\varepsilon})dx_{1}dt +\alpha_{\varepsilon}.  \label{44}
\end{gather}
}}

\textit{\textrm{Now, we consider all integrals over $l^{T}_{\varepsilon}$
included into variational inequality 
\begin{equation}  \label{45}
\beta(\varepsilon)\sum \limits_{j\in \Upsilon_{\varepsilon}}\int
\limits_{0}^{T}\int
\limits_{l^{j}_{\varepsilon}}(\partial_{t}H^{j}_{\varepsilon,\phi}+\mathcal{L%
}H^{j}_{\varepsilon,\phi}+\sigma(H^{j}_{\varepsilon,\phi})- \mathcal{L}%
\phi(x,t))(H^{j}_{\varepsilon,\phi}-u_{\varepsilon})dx_{1}dt\equiv \mathcal{P%
}_{\varepsilon},
\end{equation}
where $\mathcal{L}=\frac{\pi}{2\alpha^{2}C_{0}l_{0}}$. }}

\textit{\textrm{As $\vert \phi(x, t) - \phi(P^j_\varepsilon, t)\vert \le K
a_\varepsilon$ for $x\in l^j_\varepsilon$, we have 
\begin{equation}  \label{46}
\mathcal{P}_{\varepsilon}=\beta(\varepsilon)\sum \limits_{j\in
\Upsilon_{\varepsilon}}\int \limits_{0}^{T}\int
\limits_{l^{j}_{\varepsilon}}(\partial_{t}H^{j}_{\varepsilon,\phi}+\mathcal{L%
}H^{j}_{\varepsilon,\phi}+\sigma(H^{j}_{\varepsilon,\phi})- \mathcal{L}%
\phi(P^{j}_{\varepsilon},t))(H^{j}_{\varepsilon}-u_{\varepsilon})dx_{1}dt
+\alpha_{\varepsilon},
\end{equation}
where $\alpha_{\varepsilon}\to 0$ as $\varepsilon \to 0$. Taking into
account that $H^{j}_{\varepsilon,\phi}(t)$ is a solution of the problem %
\eqref{25}, and using that $u_{\varepsilon}\Bigl|_{l^{T}_{\varepsilon}}\ge 0$%
, we conclude that $\mathcal{P}_{\varepsilon}\le 0$. }}

\textit{\textrm{Using the last conclusion, from \eqref{38}, we obtain the
inequality 
\begin{gather}
\int \limits_{Q^{T}}\nabla \phi
\nabla(\phi-Q_{\varepsilon,\phi}-u_{\varepsilon})dx{dt}+\frac{4}{\alpha^{2}}%
\sum \limits_{j\in \Upsilon_{\varepsilon}}\int \limits_{0}^{T}\int
\limits_{(\partial{T}^{j}_{\varepsilon/4})^{+}}(\phi-H^{j}_{\varepsilon,%
\phi}(t))(\phi-u_{\varepsilon})ds{dt}\ge  \notag \\
\ge \int \limits_{Q^{T}}f(\phi-Q_{\varepsilon,\phi}-u_{\varepsilon})dx{dt}%
+\alpha_{\varepsilon}.  \label{47}
\end{gather}
}}

\textit{\textrm{The assertion proved in \cite{D3GCPodSh2018},\cite{Zubova7}
implies that 
\begin{gather}
\lim \limits_{\varepsilon \to 0}\frac{4}{\alpha^{2}}\sum \limits_{j\in
\Upsilon_{\varepsilon}}\int \limits_{0}^{T}\int \limits_{(\partial{T}%
^{j}_{\varepsilon/4})^{+}}(\phi(P^{j}_{\varepsilon},t)-H^{j}_{\varepsilon,%
\phi}(t))(\phi-u_{\varepsilon})ds{dt} =  \notag \\
=\frac{\pi}{\alpha^{2}}\int \limits_{0}^{T}\int
\limits_{\Gamma_{2}}(\phi(x,t)-H_{\phi}(x,t))(\phi-u_{\varepsilon})dx_{1}dt,
\label{48}
\end{gather}
where $H_{\phi}(x,t)$ is defined by \eqref{14}. }}

\textit{\textrm{Using the properties of the function $Q_{\varepsilon ,\phi }$%
, from \eqref{47}, \eqref{48}, we conclude, as $\varepsilon \rightarrow 0$,
that $u_{0}$ satisfies the following variational inequality 
\begin{equation}
\int \limits_{Q^{T}}\nabla \phi \nabla (\phi -u_{0})dx{dt}+\frac{\pi }{%
\alpha ^{2}}\int \limits_{\Gamma _{2}^{T}}(\phi -H_{\phi })(\phi
-u_{0})dx_{1}dt\geq \int \limits_{Q^{T}}f(\phi -u_{0})dx{dt},  \label{49}
\end{equation}%
where $\phi $ is an arbitrary function from $L^{2}(0,T;H^{1}(\Omega ,\Gamma
_{1}))$, and $H_{\phi }(x,t)\in H^{1}(0,T;L^{2}(\Gamma _{2}))$ is a solution
of the obstacle problem for a.e. $x\in \Gamma _{2}$ 
\begin{equation*}
\left \{ 
\begin{array}{lr}
\frac{d}{dt}H_{\phi }+\mathcal{L}H_{\phi }+\sigma (H_{\phi })\geq \mathcal{L}%
\phi ,\,H_{\phi }\geq 0, & t\in (0,T), \\ 
H_{\phi }(\frac{d}{dt}H_{\phi }+\mathcal{L}H_{\phi }+\sigma (H_{\phi })-%
\mathcal{L}\phi )=0, & t\in (0,T), \\ 
H_{\phi }(x,0)=u^{0}(x). & 
\end{array}%
\right.
\end{equation*}%
}}

\section{The time periodic problem}

\subsection{Time periodic solutions of the microscopic problem}

In order to show the existence and uniqueness of time-periodic solutions of $%
(PP)$ it is useful to reformulate the problem in terms of the abstract
subdifferentials of convex functions on Hilbert space $H$

\begin{equation}
\left \{ 
\begin{array}{lc}
\frac{du}{dt}(t)+\partial \phi ^{t}(u(t))\ni 0 & \text{on }H \\ 
u(0)=u_{0}. & 
\end{array}%
\right.  \label{EcAbstract_Formulation}
\end{equation}%
Obviously, to apply such abstract formulation, since the term $\beta
(\varepsilon )\partial _{t}u_{\varepsilon } $ appears in the dynamic
boundary condition, we must introduce a change of time-variable 
\begin{equation}
t = \widetilde{t}_\varepsilon \beta (\varepsilon )\text{ and }\widetilde{u}%
_\varepsilon(x, \widetilde{t}_\varepsilon) = u_\varepsilon (x, \widetilde{t}%
_\varepsilon \beta (\varepsilon )) = u_\varepsilon (x, t).
\label{ChageTimevariable}
\end{equation}%
In this way, 
\begin{equation*}
\beta (\varepsilon )\frac{\partial u_\varepsilon}{\partial t}(x,t)=\frac{%
\partial \widetilde{u}_\varepsilon}{\partial \widetilde{t}_\varepsilon}(x, 
\widetilde{t}_\varepsilon),
\end{equation*}%
and thus we can apply the abstract theory to the function $\widetilde{u}%
_\varepsilon(x, \widetilde{t}_\varepsilon)$. To simplify the notation we
dispense with that distinction and identify $u_{\varepsilon }$ with $%
\widetilde{u}_\varepsilon$.

Our key point of view is based on the main theorem of \cite{D-Jimenez} (see
also Remark 3.1 of \cite{BDVrabie}). When applied to our formulation of $%
(IVP)$, we get the following: let $H=L^{2}(l_{\varepsilon })$, and $\phi
^{t}:H\rightarrow (-\infty ,+\infty ]$, with $\phi ^{t}\not \equiv +\infty $, 
$D(\phi ^{t})=\left \{ u\in H,\phi ^{t}(u)<+\infty \right \} $ is the function
given by%
\begin{equation*}
\phi ^{t}(u)=\left \{ 
\begin{array}{cc}
\begin{gathered} \frac{1}{2}\int \limits_{\Omega }\Vert \nabla U(x)\Vert
^{2}dx-\int \limits_{\Omega }f(x,t)U(x)dx + \beta (\varepsilon
)\int \limits_{l_{\varepsilon }}j_{\sigma }(u(\sigma ))d\sigma +
\int \limits_{l_{\varepsilon }}j_{S}(u(\sigma ))d\sigma , \\ \text{if }U\in
H^{1}(\Omega ),U\vert_{l_{\varepsilon }} = u,\partial_{\nu }U\vert_{\gamma
_{\varepsilon }} = 0, U\vert_{\Gamma _{1}}=0,j_{\sigma }(u),j_{S}(u)\in
L^{1}(l_{\varepsilon }),\end{gathered} &  \\ 
+\infty \  \  \text{in the rest}, & 
\end{array}%
\right. 
\end{equation*}%
where the subdifferentials of the convex functions $j_{\sigma }$ and $j_{S}$
are given by $\partial j_{\sigma }(r)=\sigma (r)$ and $\partial
j_{S}(r)=\theta (r)$, the Signorini maximal monotone graph in $\mathbb{R}^{2}
$ given by 
\begin{equation*}
\theta (r)=\left \{ 
\begin{array}{cl}
\{0\}, & r>0, \\ 
\lbrack 0,+\infty ), & r=0, \\ 
\emptyset , & r<0,%
\end{array}%
\right. 
\end{equation*}%
i.e. the functions are 
\begin{equation*}
j_{\sigma }(r)=\int_{0}^{r}\sigma (s)ds\text{ and }j_{S}(r)=0\text{ if }%
r\geq 0,j_{S}(r)=+\infty \text{ if }r<0.\text{ }
\end{equation*}%
$\phi ^{t}$ is convex lower semi-continuous and $\phi ^{t}\not \equiv +\infty 
$ on $H=L^{2}(l_{\varepsilon }).$ As in \cite{D-Jimenez}, it is easy to
prove that $\overline{D(\phi ^{t})}=L_{+}^{2}(l_{\varepsilon })=\left \{ w\in
L^{2}(l_{\varepsilon }):w\geq 0\text{ on }l_{\varepsilon }\right \} .$
Moreover, after the change of variables (\ref{ChageTimevariable}) the
initial value problem (IVP) 
can be formulated in terms of the abstract theory for subdifferentials of
convex functions (\ref{EcAbstract_Formulation}). In addition, we have:

\begin{lemma}
\label{LemmHypoKenmoch}Assuming $f\in H^{1}(0,T;L^{2}(\Omega ))$, $\phi
^{t}(u)$ satisfies the Kenmochi assumption (see Section 1.5 of \cite%
{Kenmochi1981}): for any $r>0$ $\exists \alpha _{r}\in H^{1}(0,T)$ and $%
\beta _{r}\in W^{1,1}(0,T)$ such that for each $s,t\in \lbrack 0,T]$ with $%
s\leq t$ and $z\in D(\phi ^{s})$ with $\left \vert z\right \vert _{H}\leq r$
there exists $\widehat{z}\in D(\phi ^{t})$ such that%
\begin{equation}
\left \vert z-\widehat{z}\right \vert _{H}\leq \left \vert \alpha
_{r}(t)-\alpha _{r}(s)\right \vert (1+\left \vert \phi ^{s}(z)\right \vert
_{H}^{1/2})  \label{HypoKen1}
\end{equation}

and 
\begin{equation}
\phi ^{t}(\widehat{z})-\phi ^{s}(z)\leq \left \vert \beta _{r}(t)-\beta
_{r}(s)\right \vert (1+\left \vert \phi ^{s}(z)\right \vert _{H}).
\label{HypoKen2}
\end{equation}
\end{lemma}

\noindent \textit{Proof.} Let $r>0$, $s,t\in \lbrack 0,T]$ with $s\leq t$,
and let $z\in D(\phi ^{s}),$ $\left \vert z\right \vert _{H}\leq r.$ Since
the only time dependence of $\phi ^{t}$ is contained in the linear term $%
\int_{\Omega }f(x,t)U(x)\,dx,$ the effective domain is independent of $t$: 
\begin{equation*}
D(\phi ^{t})=D(\phi ^{s}),\qquad \forall \,s,t\in \lbrack 0,T].
\end{equation*}%
Therefore, to check both conditions in the statement, we simply choose $%
\widehat{z}=z.$ \ For the first condition (\ref{HypoKen1}), we have $%
\left
\vert z-\widehat{z}\right \vert _{H}=0.$ Hence, it is enough to take $%
\alpha _{r}(t)\equiv 0$ and (\ref{HypoKen1}) holds. To verify (\ref{HypoKen2}%
), let $U\in H^{1}(\Omega )$ be the extension associated with $z$. Since $%
\widehat{z}=z$, 
\begin{equation*}
\phi ^{t}(\widehat{z})-\phi ^{s}(z)=\phi ^{t}(z)-\phi
^{s}(z)=-\int \limits_{\Omega }(f(x,t)-f(x,s))U(x)\,dx.
\end{equation*}%
Therefore, we get the estimate 
\begin{equation*}
|\phi ^{t}(z)-\phi ^{s}(z)|\leq \Vert f(\cdot, t)-f(\cdot, s)\Vert _{L^{2}(\Omega )}\Vert
U\Vert _{L^{2}(\Omega )}.
\end{equation*}%
Since $U=0$ on $\Gamma _{1}$, Poincar\'{e}'s inequality yields 
\begin{equation*}
\Vert U\Vert _{L^{2}(\Omega )}\leq C_{P}\Vert \nabla U\Vert _{L^{2}(\Omega
)}.
\end{equation*}%
Moreover, 
\begin{align*}
\phi ^{s}(z)& =\frac{1}{2}\Vert \nabla U\Vert _{L^{2}(\Omega
)}^{2}-\int \limits_{\Omega }f(x,s)U(x)\,dx \\
& \quad + \beta (\varepsilon )\int \limits_{l_{\varepsilon }}j_{\sigma
}(z)\,d\sigma +\int \limits_{l_{\varepsilon }}j_{S}(z)\,d\sigma .
\end{align*}%
Since the last two terms are nonnegative, 
\begin{equation*}
\phi ^{s}(z)\geq \frac{1}{2}\Vert \nabla U\Vert _{L^{2}(\Omega )}^{2}-\Vert
f(\cdot, s)\Vert _{L^{2}(\Omega )}\Vert U\Vert _{L^{2}(\Omega )}.
\end{equation*}%
Using Poincar\'{e}'s inequality and Young's inequality, 
\begin{equation*}
\phi ^{s}(z)\geq \frac{1}{4}\Vert \nabla U\Vert _{L^{2}(\Omega )}^{2}-C,
\end{equation*}%
where $C>0$ depends only on 
\begin{equation*}
\sup_{\tau \in \lbrack 0,T]}\Vert f(\cdot, \tau )\Vert _{L^{2}(\Omega )}.
\end{equation*}%
Hence 
\begin{equation*}
\Vert \nabla U\Vert _{L^{2}(\Omega )}^{2}\leq C\bigl(1+\phi ^{s}(z)\bigr),
\end{equation*}%
and consequently 
\begin{equation*}
\Vert U\Vert _{L^{2}(\Omega )}\leq C\bigl(1+\phi ^{s}(z)\bigr)^{1/2}.
\end{equation*}%
Substituting into the previous estimate, 
\begin{equation*}
|\phi ^{t}(z)-\phi ^{s}(z)|\leq C\, \Vert f(\cdot, t)-f(\cdot, s)\Vert _{L^{2}(\Omega )}%
\bigl(1+\phi ^{s}(z)\bigr)^{1/2}.
\end{equation*}%
Since 
\begin{equation*}
(1+a)^{1/2}\leq 1+a,\text{ for any }a\geq 0,
\end{equation*}%
we obtain 
\begin{equation*}
|\phi ^{t}(z)-\phi ^{s}(z)|\leq C\, \Vert f(\cdot, t)-f(\cdot, s)\Vert _{L^{2}(\Omega )}%
\bigl(1+\phi ^{s}(z)\bigr).
\end{equation*}%
Since $f\in H^{1}(0,T;L^{2}(\Omega )),$ we have 
\begin{equation*}
\Vert f(\cdot, t)-f(\cdot, s)\Vert _{L^{2}(\Omega )}\leq \int_{s}^{t}\Vert f^{\prime
}(\cdot, \tau )\Vert _{L^{2}(\Omega )}\,d\tau .
\end{equation*}%
Define 
\begin{equation*}
\beta _{r}(t)=C\int_{0}^{t}\Vert f^{\prime }(\cdot, \tau )\Vert _{L^{2}(\Omega
)}\,d\tau .
\end{equation*}%
Then $\beta _{r}\in W^{1,1}(0,T)$ and 
\begin{equation*}
\phi ^{t}(\widehat{z})-\phi ^{s}(z)\leq |\beta _{r}(t)-\beta _{r}(s)|\bigl(%
1+\phi ^{s}(z)\bigr),
\end{equation*}%
and (\ref{HypoKen2}) holds.$_{\blacksquare }$

\bigskip

Then we have

\begin{theorem}
Assume $f\in H^{1}(\mathbb{R},L^{2}(\Omega ))$ $T$-periodic. Then, there
exists a unique strong $T$-periodic solution of $(PP).$
\end{theorem}

\noindent \textit{Proof.} \ The existence of a $T$-periodic solution of $%
(PP) $ is a consequence of Theorem 2.3.1 of \cite{Kenmochi1981} (thanks to
the conditions (\ref{HypoKen1}) and (\ref{HypoKen2})). The uniqueness is
consequence from the fact that $\phi ^{t}(u)$ is strictly convex and Theorem
2.3.2 of \cite{Kenmochi1981}.$_{\blacksquare }$

\bigskip


\textbf{Remark. } Let $\Omega \subset \mathbb{R}^{n}$ be an open set
(bounded or not) and $T>0$. We consider the Bochner-Sobolev space 
\begin{equation*}
H^{1}(0,T;L^{2}(\Omega )):=\bigl \{f:(0,T)\times \Omega \rightarrow \mathbb{R}%
\; \vert \;f\in L^{2}(0,T;L^{2}(\Omega )),\; \; \partial _{t}f\in
L^{2}(0,T;L^{2}(\Omega ))\bigr \}.
\end{equation*}%
Here, $\partial _{t}f$ is taken in the distributional sense. The norm is 
\begin{equation*}
\Vert f\Vert _{H^{1}(0,T;L^{2}(\Omega ))}^{2}=\int_{0}^{T}\int_{\Omega
}|f(x, t)|^{2}\,dx\,dt+\int_{0}^{T}\int_{\Omega }|\partial _{t}f(x,
t)|^{2}\,dx\,dt.
\end{equation*}

For almost every $x\in \Omega $, the map $t\mapsto f(x, t)$ belongs to $%
H^{1}(0,T)$, and hence to $C([0,T])$ by the Sobolev embedding in one
dimension. Therefore, we can define 
\begin{equation*}
M(x):=\sup_{t\in \lbrack 0,T]}|f(x, t)|.
\end{equation*}

\begin{lemma}
If $f\in H^{1}(0,T;L^{2}(\Omega ))$ then $\sup_{t\in \lbrack 0,T]}|f(\cdot,
t)|\in L^{2}(\Omega )$.
\end{lemma}

For the proof, it is useful to recall a well-known result (here it is given
with an explicit estimate)

\begin{lemma}[Morrey inequality in 1D]
For any $g\in H^{1}(0,T)$, we have 
\begin{equation}
\Vert g\Vert _{L^{\infty }(0,T)}\leq C_{T}\Vert g\Vert _{H^{1}(0,T)},
\label{Estimate Morrey}
\end{equation}%
with $C_{T}=\sqrt{\frac{2}{T}+2T}$.
\end{lemma}

\noindent \textit{Proof of the estimate (\ref{Estimate Morrey}). }For any $%
s,t\in \lbrack 0,T]$, 
\begin{equation*}
g(t)=g(s)+\int_{s}^{t}g^{\prime }(\tau )\,d\tau ,
\end{equation*}%
hence $|g(t)|\leq |g(s)|+\sqrt{T}\Vert g^{\prime }\Vert _{L^{2}(0,T)}$.
Averaging over $s\in \lbrack 0,T]$ gives 
\begin{equation*}
|g(t)|\leq \frac{1}{T}\int_{0}^{T}|g(s)|\,ds+\sqrt{T}\Vert g^{\prime }\Vert
_{L^{2}(0,T)}.
\end{equation*}%
By Cauchy-Schwarz, 
\begin{equation*}
\frac{1}{T}\int_{0}^{T}|g(s)|\,ds\leq \frac{1}{\sqrt{T}}\Vert g\Vert
_{L^{2}(0,T)}.
\end{equation*}%
Thus for every $t$, 
\begin{equation*}
|g(t)|\leq \frac{1}{\sqrt{T}}\Vert g\Vert _{L^{2}(0,T)}+\sqrt{T}\Vert
g^{\prime }\Vert _{L^{2}(0,T)}.
\end{equation*}%
Taking the supremum over $t$ and using $(a+b)^{2}\leq 2(a^{2}+b^{2})$ yields 
\begin{equation*}
\Vert g\Vert _{L^{\infty }(0,T)}^{2}\leq 2\left( \frac{1}{T}\Vert g\Vert
_{L^{2}(0,T)}^{2}+T\Vert g^{\prime }\Vert _{L^{2}(0,T)}^{2}\right) \leq
2\max \left \{ \frac{1}{T},T\right \} \Vert g\Vert _{H^{1}(0,T)}^{2}.
\end{equation*}%
So $C_{T}=\sqrt{2\max \{1/T,T\}}$ works, but the simpler constant $\sqrt{%
2/T+2T}$ is also valid (since $\max \{1/T,T\} \leq \frac{1}{T}+T$ when $T\geq
1$, but the inequality $(a+b)^{2}\leq 2a^{2}+2b^{2}$ is uniform). Actually
from the pointwise bound: 
\begin{equation*}
\sup_{t}|g(t)|\leq \frac{1}{\sqrt{T}}\Vert g\Vert _{L^{2}}+\sqrt{T}\Vert
g^{\prime }\Vert _{L^{2}}
\end{equation*}%
and by Cauchy-Schwarz in $\mathbb{R}^{2}$, 
\begin{equation*}
\frac{1}{\sqrt{T}}\Vert g\Vert _{L^{2}}+\sqrt{T}\Vert g^{\prime }\Vert
_{L^{2}}\leq \sqrt{\frac{1}{T}+T}\; \sqrt{\Vert g\Vert _{L^{2}}^{2}+\Vert
g^{\prime }\Vert _{L^{2}}^{2}}=\sqrt{\frac{1}{T}+T}\; \Vert g\Vert _{H^{1}}.
\end{equation*}%
Hence $C_{T}=\sqrt{T+\frac{1}{T}}$.$_{\blacksquare }$

\bigskip

\noindent \textit{Proof of Lemma 5. } Fix $x\in \Omega $ such that $%
f(x,\cdot )\in H^{1}(0,T)$ (this holds for almost every $x$). Applying the
Morrey Lemma with $g(t)=f(x,t)$, we obtain, for every $t\in \lbrack 0,T]$, 
\begin{equation*}
|f(x,t)|\leq \sqrt{T+\frac{1}{T}}\; \left(
\int \limits_{0}^{T}|f(x,s)|^{2}ds\right) ^{1/2}+\sqrt{T+\frac{1}{T}}\; \left(
\int \limits_{0}^{T}|\partial _{s}f(x,s)|^{2}ds\right) ^{1/2}.
\end{equation*}%
Taking the supremum in $t$ gives 
\begin{equation*}
M(x)\leq C_{T}\left( \int \limits_{0}^{T}|f(x,s)|^{2}ds\right)
^{1/2}+C_{T}\left( \int \limits_{0}^{T}|\partial _{s}f(x,s)|^{2}ds\right)
^{1/2},
\end{equation*}%
with $C_{T}=\sqrt{T+\frac{1}{T}}$.

\noindent Now, let us find an estimate in $L^{2}(\Omega )$ of $M.$ Square
the inequality and use $(a+b)^{2}\leq 2(a^{2}+b^{2})$: 
\begin{equation*}
M(x)^{2}\leq
2C_{T}^{2}\int \limits_{0}^{T}|f(x,s)|^{2}ds\;+\;2C_{T}^{2}\int%
\limits_{0}^{T}|\partial _{s}f(x,s)|^{2}ds.
\end{equation*}%
Now integrate over $\Omega $: 
\begin{equation*}
\int \limits_{\Omega }M(x)^{2}dx\leq 2C_{T}^{2}\int \limits_{\Omega
}\int \limits_{0}^{T}|f(x,s)|^{2}ds\,dx + 2C_{T}^{2}\int \limits_{\Omega
}\int \limits_{0}^{T}|\partial _{s}f(x,s)|^{2}ds\,dx.
\end{equation*}%
By Tonelli's theorem (all integrands are nonnegative), we may swap the
order: 
\begin{equation*}
\int \limits_{\Omega }M(x)^{2}dx\leq
2C_{T}^{2}\int \limits_{0}^{T}\int \limits_{\Omega }|f(x,s)|^{2}dx\,ds +
2C_{T}^{2}\int \limits_{0}^{T}\int \limits_{\Omega }|\partial
_{s}f(x,s)|^{2}dx\,ds.
\end{equation*}%
But these two terms are exactly $2C_{T}^{2}$ times the two summands of the $%
H^{1}$ norm: 
\begin{equation*}
\int \limits_{\Omega }M(x)^{2}dx\leq 2C_{T}^{2}\Vert f\Vert
_{L^{2}(0,T;L^{2}(\Omega ))}^{2} + 2C_{T}^{2}\Vert \partial _{t}f\Vert
_{L^{2}(0,T;L^{2}(\Omega ))}^{2}.
\end{equation*}%
Thus 
\begin{equation*}
\Vert M\Vert _{L^{2}(\Omega )}^{2}\leq 2C_{T}^{2}\Vert f\Vert
_{H^{1}(0,T;L^{2}(\Omega ))}^{2}.
\end{equation*}%
This completes the proof.$_{\blacksquare }$

\begin{theorem}
\label{Thm PeriodicEstimates}Assume 
\begin{equation*}
f\in H^{1}(\mathbb{R}; L^{2}(\Omega ))\text{, }T-\text{time-periodic}.
\end{equation*}%
For $i=1,2$, define $f_{i}\in L^{2}(\Omega )$, by 
\begin{equation*}
f_{1}(x):=\underset{t\in \lbrack 0,T]}{\essmin}f(x,t)\leq f_{2}(x):=\underset%
{t\in \lbrack 0,T]}{\essmax}f(x,t).
\end{equation*}%
Let the $u_{\varepsilon }^{i}(x)$ be (unique) solutions of the stationary
problems 
\begin{equation}
(SP)\left \{ 
\begin{array}{lr}
-\Delta _{x}u_{\varepsilon }^{i}(x)=f_{i}(x), & x,\in \Omega , \\ 
u_{\varepsilon }^{i}\geq 0,\partial _{\nu }u_{\varepsilon }^{i}+\beta
(\varepsilon )\sigma (u_{\varepsilon }^{i})\geq 0, &  \\ 
u_{\varepsilon }^{i}(\partial _{\nu }u_{\varepsilon }^{i}+\beta (\varepsilon
)\sigma (u_{\varepsilon }^{i}))=0, & x\in l_{\varepsilon }, \\ 
\partial _{\nu }u_{\varepsilon }^{i}=0, & x\in \gamma _{\varepsilon }, \\ 
u_{\varepsilon }^{i}(x)=0, & x\in \Gamma _{1}.%
\end{array}%
\right. 
\end{equation}%
Then, we have%
\begin{equation}
u_{\varepsilon }^{1}(x)\leq u_{\varepsilon }^{T}(x,t)\leq u_{\varepsilon
}^{2}(x)\text{ for any }t\in \mathbb{R},\text{ on }l_{\varepsilon }.
\label{eqComparisonStationary}
\end{equation}
\end{theorem}

\noindent \textit{Proof.} \ We point out that, thanks to Lemma \ref%
{LemmHypoKenmoch} we can apply Theorem 1.1.2 of \cite{Kenmochi1981} and
then, for any $u^{0}\in L_{+}^{2}(l_{\varepsilon })$ there exists a unique
weak solution $u\in C([0,T]; L^{2}(l_{\varepsilon }))$ of $(IVP)$ (i.e. of (%
\ref{EcAbstract_Formulation}) associated to the convex function $\phi ^{t}(u)
$ defined above). Moreover, for any $t\in (0,T],$ such a solution satisfies
(the regularizing effect) that $u(\cdot,t)\in D(\phi ^{t}).$ We will get the
estimate (\ref{eqComparisonStationary}) by finding a fixed point function $%
u^{0}$ of the Poincar\'{e} map $F:K\rightarrow K,$ with $K\subset
L^{2}(l_{\varepsilon })$ a closed and convex set, 
\begin{equation*}
F(u^{0}(\cdot)) = u_{\varepsilon }(\cdot,T)\text{.}
\end{equation*}%
As in (\cite{Badii-Diaz1999}), we will apply the Schauder fixed point
theorem on the space $L^{2}(l_{\varepsilon })$. We need to select a closed
and convex set $K\subset L^{2}(l_{\varepsilon })$ such that i) $F(K)\subset
K $; ii) $F\vert_{K}$ is continuous, iii) $F(K)$ is relatively compact in $%
L^{2}(l_{\varepsilon })$. As the set $K$, we will take the \textit{interval} 
$[u_{\varepsilon }^{1}(\cdot),u_{\varepsilon }^{2}(\cdot)]$ $=\{w\in
L^{2}(l_{\varepsilon }),$ $u_{\varepsilon }^{1}(x)\leq w(x)\leq
u_{\varepsilon }^{2}(x)$ $a.e.$ $x\in l_{\varepsilon }\}$. Obviously $K$ is
a closed, convex and non empty set of $L^{2}(l_{\varepsilon })$. It is
useful to prove property i) as a separate auxiliary result.

\begin{lemma}
Let $f(x,t)$ and $f_{i}(x)$, $i=1,2,$ as in Theorem \ref{Thm
PeriodicEstimates}. Let $u_{\varepsilon }(x,t)$ be the unique solution of $%
(IVP)$ corresponding to an initial datum $u^{0}\in K$. Then, $u_{\varepsilon
}(\cdot,t)\in K$ for any $t\in \lbrack 0,T].$
\end{lemma}

\noindent \textit{Proof.} By the continuous dependence of solutions of (\ref%
{EcAbstract_Formulation}) with respect to the initial datum (see, e.g.
Theorem 1.1.1 of \cite{Kenmochi1981}) in order to get the comparisons
included in the information $u_{\varepsilon }(\cdot,t)\in K$ it is enough to
assume that $u^{0}\in D(\phi ^{t}).$ Let us prove that $u_{\varepsilon
}(x,t)\leq u_{\varepsilon }^{2}(x)$ (the proof of the other comparison $%
u_{\varepsilon }(x,t)\geq u_{\varepsilon }^{1}(x)$ is similar). We subtract
the differential equations for $u_{\varepsilon }(x,t)$ and for $%
u_{\varepsilon }^{2}(x)$, multiply by $\left[ u_{\varepsilon
}(x,t)-u_{\varepsilon }^{2}(x)\right] _{+}=\max (u_{\varepsilon
}(x,t)-u_{\varepsilon }^{2}(x),0)$ and integrate by parts. After using the
boundary conditions and the monotonicity of the nonlinear terms $\partial
j_{\sigma }(r)=\sigma (r)$ and $\partial j_{S}(r) = \theta(r)$, we arrive to
the inequality 
\begin{equation*}
\begin{array}{c}
\int \limits_{\Omega }\left \vert \nabla \left[ u_{\varepsilon
}(x,t)-u_{\varepsilon }^{2}(x)\right]_{+}\right \vert ^{2}dx + \frac{d}{dt}%
\int \limits_{l_{\varepsilon }}\left \vert \left[ u_{\varepsilon }(s, t) -
u_{\varepsilon }^{2}(s)\right] _{+}\right \vert ^{2}ds \\ 
\leq \int \limits_{\{x\in \Omega \vert u_{\varepsilon }(x,t) > u_{\varepsilon
}^{2}(x)\}}(f(x,t)-f^{2}(x))(u_{\varepsilon }(x,t) - u_{\varepsilon
}^{2}(x))_{+}dx.%
\end{array}%
\end{equation*}%
Then, since $f(x,t)\leq f^{2}(x)$ on $\Omega $, we get that 
\begin{equation*}
\int \limits_{l_{\varepsilon }}\left \vert \left[ u_{\varepsilon
}(s,t)-u_{\varepsilon }^{2}(s)\right] _{+}\right \vert ^{2}ds\leq
\int \limits_{l_{\varepsilon }}\left \vert \left[ u^{0}(s)-u_{\varepsilon
}^{2}(s)\right] _{+}\right \vert ^{2}ds,
\end{equation*}%
and since $u^{0}\in K$, we get that $u_{\varepsilon }(\cdot,t)\leq
u_{\varepsilon }^{2}(\cdot)$ on $l_{\varepsilon }._{\blacksquare }$

\noindent \textit{Proof of Theorem \ref{Thm PeriodicEstimates}
(continuation).} \ The above Lemma shows that $F([u_{\varepsilon
}^{1},u_{\varepsilon }^{2}])\subset $ $[u_{\varepsilon }^{1},u_{\varepsilon
}^{2}]$. Let us check that $F|_{K}$ is continuous (property ii)). It is a
trivial consequence of the $L^{2}(l_{\varepsilon })$ estimate 
\begin{equation*}
\int_{l_{\varepsilon }}\left \vert u_{\varepsilon }(s,t)-u_{\delta
}(s,t)\right \vert ^{2}ds\leq \int_{l_{\varepsilon }}\left \vert
u^{0}(s)-u_{\delta }^{0}(s)\right \vert ^{2}ds,
\end{equation*}%
which holds, thanks to Theorem1.1.1 of \cite{Kenmochi1981} if $u_{\delta
}(x,t)$ denotes the solution of 
\begin{equation}
\left \{ 
\begin{array}{lc}
\frac{du}{dt}(t)+\partial \phi ^{t}(u(t))\ni 0 & \text{on }H \\ 
u(0)=u_{\delta }^{0}, & 
\end{array}%
\right. 
\end{equation}%
and $u_{\delta }^{0}\rightarrow u^{0}$ in $L^{2}(l_{\varepsilon })$.
Finally, to prove that $F(K)$ is relatively compact in $L^{2}(l_{\varepsilon
})$ (property (iii)) it is enough to observe that by the regularization
effect (see Theorem 2.1.1 of \cite{Kenmochi1981}) we have 
$u(T)\in D(\phi ^{T})\subset H^{1/2}(l_{\varepsilon })$, and since the
inclusion $H^{1/2}(l_{\varepsilon })\subset L^{2}(l_{\varepsilon })$ is
compact, we get that $F(K)$ is relatively compact in $L^{2}(l_{\varepsilon })
$. Then, by the Schauder fixed point theorem, there exists a fixed point $%
u_{\varepsilon }^{T}\in K$ of the Poincar\'{e} map $F$ and the proof is
complete. $_{\blacksquare }$

\subsection{Time periodic solutions of the homogenized problem}

Applying the techniques of the homogenization result for the initial value
problem $(IVP)$, we get

\begin{theorem}
Assume $f\in H^{1}(\mathbb{R},L^{2}(\Omega ))$, $T-$time-periodic, then, we
have that $u_{\varepsilon }^{T}(x,t)\rightharpoonup {u}_{0}^{T}(x,t)$ with ${%
u}_{0}^{T}(x,t)$ $T-$periodic solution of the homogenized problem of $(PP),$
given by the stationary problem (dependinhg on $t$ as a parameter)%
\begin{equation}
(PP)_{Hom}\left \{ 
\begin{array}{lr}
-\Delta _{x}{u}_{0}(x,t)=f(x,t), & (x,t)\in Q^{\infty }, \\ 
-\partial _{x_{2}}u_{0}+\mathcal{M}u_{0}=\mathcal{M}H_{u_{0}}, & (x,t)\in
\Gamma _{2}^{\infty }, \\ 
H_{u_{0}}\geq 0,\, \partial _{t}H_{u_{0}}+\mathcal{L}H_{u_{0}}+\sigma
(H_{u_{0}})\geq \mathcal{L}u_{0}, & (x,t)\in \Gamma _{2}^{\infty }, \\ 
H_{u_{0}}(\partial _{t}H_{u_{0}}+\mathcal{L}H_{u_{0}}+\sigma (H_{u_{0}})-%
\mathcal{L}u_{0})=0, & (x,t)\in \Gamma _{2}^{\infty }, \\ 
H_{u_{0}}(x,t)=H_{u_{0}}(x,t+T), & \text{for any }t\in \mathbb{R}\text{, }%
x\in \Gamma _{2}, \\ 
u_{0}(x,t)=0, & (x,t)\in \Gamma _{1}^{\infty },%
\end{array}%
\right.
\end{equation}%
where \textit{\textrm{$\mathcal{M}$}} $=\frac{\pi }{\alpha ^{2}}$, \textit{%
\textrm{$\mathcal{L}$}}$=\frac{\pi }{2C_{0}l_{0}\alpha ^{2}}$. Moreover, if $%
{u}_{0}^{i}(x)$ is the homogenized stationary problem associated to the data 
$f^{i}$ (see \cite{DGSH2Book}) we have 
\begin{equation}
{u}_{0}^{1}(x)\leq {u}_{0}^{T}(x,t)\leq {u}_{0}^{2}(x)\text{ a.e. }x\in
\Omega \text{, for }t\in (0,T).  \label{eqFinal comparison}
\end{equation}
\end{theorem}

\noindent \textit{Proof.} \ It is an obvious adaptation of the proof of
Theorem~\ref{main hom theorem} to this setting and we escape the details.
Moreover, if $\phi $ is an arbitrary function from\textit{\textrm{\ $L^{2}(%
\mathbb{R}; H^{1}(\Omega ,\Gamma _{1}))$, }}$T-$periodic,\textit{\textrm{\ }}%
$\phi \geq 0$ on $\Omega$ for a.e. $t$, then, $u_{\varepsilon }^{1}(x)\phi
(x,t)\leq u_{\varepsilon }^{T}(x,t)\phi (x,t)\leq u_{\varepsilon
}^{2}(x)\phi (x,t)$ for any $t\in \mathbb{R},$ and by the weak convergence
we get that 
\begin{equation*}
{u}_{0}^{1}(x)\phi (x,t)\leq {u}_{0}^{T}(x,t)\phi (x,t)\leq {u}%
_{0}^{2}(x)\phi (x,t)\text{ }
\end{equation*}%
for any arbitrary test function, which implies (\ref{eqFinal comparison}). 
\textit{\textrm{\ }}$_{\blacksquare }$

\begin{remark}
Finally, by using the techniques of \cite{DPodoSh RACSAM2024} it seems
possible to prove that if the forcing $f(x,t)$ is strongly negative, then,
the transport becomes impossible: the trace vanishes on $\Gamma _{2}^{T}$,
the membrane becomes inactive. Thus, one obtains two regimes: active
transport regime, or blocked membrane regime according to the values of $%
f(x,t)$.
\end{remark}

\bigskip

\textbf{Acknowledgement} JID was partially supported by the project
PID-2020-112517GBI00 of the AEI and MCIU/AEI/10.13039/-501100011033/FEDER,
EU.

\end{document}